\documentclass[11pt]{amsart}
\usepackage{hyperref}
\usepackage{amsfonts}
\usepackage{amsmath}
\usepackage{amssymb}
\usepackage{amscd}
\usepackage{graphicx}
\usepackage{latexsym}
\usepackage{color}
\usepackage{epsfig}
\usepackage{amsopn}
\usepackage{xypic}
\usepackage{mathrsfs}
\usepackage{array}

\usepackage{booktabs}
\usepackage{longtable}
\theoremstyle{plain}
\newtheorem{lemma}{Lemma}[section]
\newtheorem*{theorem*}{Theorem}
\newtheorem*{lemma*}{Lemma}
\newtheorem*{proposition*}{Proposition}
\newtheorem*{conjecture*}{Conjecture}
\newtheorem*{corollary*}{Corollary}
\newtheorem*{problem*}{Problem}
\newtheorem{theorem}[lemma]{Theorem}

\newtheorem{corollary}[lemma]{Corollary}
\newtheorem{proposition}[lemma]{Proposition}
\newtheorem{problem}[lemma]{Problem}

\theoremstyle{definition}
\newtheorem{definition}[lemma]{Definition}
\newtheorem{example}[lemma]{Example}
\newtheorem{remark}[lemma]{Remark}

\newtheorem{warning}[lemma]{Warning}

\newcommand{\F}[1]{\mathscr{#1}}

\newcommand{\Z}{\mathbb{Z}}
\newcommand{\N}{\mathbb{N}}

\newcommand{\C}{\mathbb{C}}
\newcommand{\Q}{\mathbb{Q}}
\newcommand{\R}{\mathbb{R}}

\newcommand{\OO}{\mathcal{O}}
\newcommand{\te}{\otimes}
\newcommand{\sm}{\setminus}

\newcommand{\id}{\mathrm{id}}

\newcommand{\cA}{\mathcal A}
\newcommand{\cM}{\mathcal M}
\newcommand{\cP}{\mathcal P}

\newcommand{\ZZ}{\mathbb{Z}}

\renewcommand{\P}{\mathbb{P}}
\newcommand{\PP}{\mathbb{P}}

\DeclareMathOperator{\ch}{ch}

\DeclareMathOperator{\Hom}{Hom}

\DeclareMathOperator{\Pic}{Pic}

\DeclareMathOperator{\NS}{NS}
\DeclareMathOperator{\Kom}{Kom}

\DeclareMathOperator{\ord}{ord}
\DeclareMathOperator{\Mov}{Mov}
\DeclareMathOperator{\Ext}{Ext}
\DeclareMathOperator{\ext}{ext}
\DeclareMathOperator{\Supp}{Supp}
\DeclareMathOperator{\ev}{ev}

\DeclareMathOperator{\sHom}{\mathcal{H} \textit{om}}
\DeclareMathOperator{\GL}{GL}
\DeclareMathOperator{\SL}{SL}

\DeclareMathOperator{\Eff}{Eff}

\DeclareMathOperator{\pare}{par}
\DeclareMathOperator{\udim}{\underline{\dim}}
\DeclareMathOperator{\edim}{edim}

\DeclareMathOperator{\Fitt}{Fitt}

\begin{document}

\date{\today}
\author[I. Coskun]{Izzet Coskun}
\author[J. Huizenga]{Jack Huizenga}
\author[M. Woolf]{Matthew Woolf}
\address{Department of Mathematics, Statistics and CS \\University of Illinois at Chicago, Chicago, IL 60607}
\email{coskun@math.uic.edu}
\email{huizenga@math.uic.edu}
\address{Department of Mathematics, Harvard University, 1 Oxford Street, Cambridge, MA 02138}
\email{mwoolf@math.harvard.edu}
\thanks{During the preparation of this article the first author was partially supported by the NSF CAREER grant DMS-0950951535, and an Alfred P. Sloan Foundation Fellowship and the second author was partially supported by a National Science Foundation Mathematical Sciences Postdoctoral Research Fellowship}

\title[Effective divisors on spaces of sheaves]{The effective cone of the moduli space of sheaves on the plane}

\begin{abstract}
Let $\xi$ be the Chern character of a stable coherent sheaf on $\PP^2$. For every $\xi$, we compute the cone of  effective divisors on the moduli space $M(\xi)$ of semistable sheaves on $\PP^2$ with Chern character $\xi$. The computation hinges on finding a good resolution of the general sheaf in $M(\xi)$. This resolution  is determined by Bridgeland stability and arises from a well-chosen Beilinson spectral sequence.  The existence of a good choice of spectral sequence depends on remarkable number-theoretic properties of the slopes of exceptional bundles.
\end{abstract}

\maketitle

\newcommand{\spacing}[1]{\renewcommand{\baselinestretch}{#1}\large\normalsize}
 \setcounter{tocdepth}{1}
\tableofcontents

\section{Introduction}\label{sec-intro}

In this paper, we determine the cone of effective divisors on any moduli space of semistable sheaves on $\PP^2$. 

The effective and movable cones are among the most important invariants of a variety $X$ (see \cite{Lazarsfeld}). They dictate both the geometry and arithmetic of $X$.   The {\em movable cone} $\Mov(X)$, which is spanned by divisor classes whose stable base locus does not contain any divisorial components,  controls the rational maps from $X$ to other projective varieties. The {\em effective cone} $\Eff(X)$, which is spanned by classes of effective divisors, plays a prominent role in determining the basic birational invariants of $X$ such as the Kodaira dimension and  is expected to dictate the asymptotics of rational points on $X$ via Manin's conjecture (see \cite{Tschinkel}).

Let $\xi$ be the Chern character of a stable coherent sheaf on $\PP^2$. Then the Gieseker moduli space $M(\xi)$, which parameterizes $S$-equivalence classes of semistable sheaves on  $\PP^2$ with Chern character $\xi$, is an irreducible, normal, factorial projective variety (see \cite{LePotierLectures}). 

When $\xi$ is the Chern character of an \emph{exceptional} or \emph{height zero} bundle (see \S \ref{sec-prelim}), $M(\xi)$ is a point or has Picard rank one, respectively. In these cases there is nothing further to discuss. Hence, we will always assume that $M(\xi)$ is not an exceptional or height zero moduli space.  There is a natural pairing on the Grothendieck group $K(\PP^2)$ that associates to two Chern characters $\xi$ and  $\zeta$ the Euler characteristic $\chi(\xi^* , \zeta)$. The Picard group of $M(\xi)$ is a free abelian group of rank two,  naturally identified with $\xi^{\perp}$ in $K(\PP^2)$ with respect to this pairing \cite{DLP}. Correspondingly, the N\'{e}ron-Severi space $\NS(M(\xi))$ is a two-dimensional vector space equal to $\Pic(M(\xi)) \otimes \R$. Therefore, in order to determine the effective and movable cones, it suffices to specify their two extremal rays.

We now describe our results on the effective cone of $M(\xi)$ in greater detail. Since $M(\xi)$ is factorial, in order to construct effective Cartier divisors, it suffices to impose codimension one geometric conditions on sheaves. Brill-Noether conditions provide a systematic method for constructing effective divisors. Let $\zeta\in \xi^\perp$ be the Chern character of a vector bundle $V$.  There are $h^0$- and $h^2$-Brill-Noether loci on $M(\xi)$ depending on whether the slope $\mu(\xi\te \zeta)$ is negative or positive. These are related by duality and are responsible for the two different extremal rays of the effective cone. For simplicity, we will assume that the slope is positive and describe the $h^0$-Brill-Noether loci first and later explain the necessary changes needed for $h^2$-Brill-Noether loci. Define $$D_V := \{ U \in M(\xi): h^0(U\te V) \not= 0\}$$ and endow it with the natural determinantal scheme structure.  The locus $D_V$  is either an effective divisor or $D_V = M(\xi)$. Consequently, if there exists $U \in M(\xi)$ such that $h^0(U \otimes V) = 0$, then $D_V$ is an effective Cartier divisor called the {\em Brill-Noether divisor associated to $V$}. 

\begin{definition}
A  sheaf $U \in M(\xi)$ is {\em cohomologically orthogonal} to a vector bundle $V$ if $h^i(U\te V) =0$ for every $i$.
\end{definition}

If $U$ is cohomologically orthogonal to $V$, then $\chi(U\te V)=0$ and, by the Riemann-Roch Theorem, the discriminant of $V$ is determined by the slope of $V$ (see \S \ref{sec-prelim}). Furthermore, the Brill-Noether divisor $D_V$ is an effective divisor that does not contain $U$ in its base locus. Hence, the following problem naturally arises in studying stable base loci of linear systems on $M(\xi)$. 

\begin{problem}[Higher rank interpolation]\label{prob-int}
Given a sheaf $U\in M(\xi)$, determine the minimum slope $\mu^+\in \Q$ such that \begin{enumerate}
\item $\mu(\xi)+\mu^+\geq 0$, and
\item $U$ is cohomologically orthogonal to a vector bundle $V$ of slope $\mu^+$.
\end{enumerate}
\end{problem}

The name of the problem is a holdover from the case where $M(\xi)$ is a Hilbert scheme of points $\P^{2[n]}$.  In that case the problem amounts to determining when a zero-dimensional scheme $Z$ can impose independent conditions on sections of a vector bundle.  This problem was previously solved in \cite{HuizengaPaper2} for a general scheme $Z$ and in \cite{CoskunHuizenga} for monomial schemes. In this paper, we  will solve the problem for general sheaves in any moduli space $M(\xi)$. This will determine the effective  cone of $M(\xi)$.

Our answer to the interpolation problem is in terms of the classification of stable vector bundles on $\PP^2$ due to Dr\'{e}zet and Le Potier \cite{DLP} and \cite{LePotierLectures}. The classification, which we will recall in \S \ref{sec-prelim}, can best be phrased in terms of the slope $\mu$ and the discriminant $\Delta$. There is a fractal-like curve $\delta(\mu)$ in the $(\mu, \Delta)$-plane such that the moduli space of vector bundles with invariants $(r, \mu, \Delta)$ is positive dimensional if and only if $\Delta \geq \delta(\mu)$ and $r$ is a sufficiently divisible integer so that the Chern class $c_1$ and the Euler characterisitc $\chi$ are integers. In addition, there are exceptional bundles whose moduli spaces are isolated points. The slopes of exceptional bundles are in one-to-one correspondence with dyadic integers (see \S \ref{sec-prelim}).  For each exceptional slope $\alpha$, there is an interval $I_{\alpha} = ( \alpha-x_{\alpha}, \alpha + x_{\alpha})$ such that the curve $\delta(\mu)$ is defined on $(\alpha - x_{\alpha}, \alpha]$ by a parabola $P_1^{\alpha}$ and on $[\alpha, \alpha + x_{\alpha})$ by a parabola $P_2^{\alpha}$ (see \S \ref{sec-prelim} for explicit equations).

Each Chern character $\xi$ determines a (possibly-degenerate, if $r(\xi)=0$) parabola $Q_{\xi}$ in the $(\mu, \Delta)$-plane parameterizing pairs $(\mu, \Delta)$ that are orthogonal to $\xi$. Assume that the positive intersection of $Q_{\xi}$ with the line $\Delta = \frac{1}{2}$ occurs along the interval $I_{\alpha}$ containing the exceptional slope $\alpha$. Our answer depends on 3 cases.

First, in special cases, the invariants of  the exceptional bundle $E_{\alpha}$ of slope $\alpha$ may lie on the parabola $Q_{\xi}$. In this case, let $(\mu^+, \Delta^+)$ be the slope and discriminant of $E_\alpha$. Second, if the invariants of $E_{\alpha}$ are below the parabola $Q_{\xi}$, let $(\mu^+, \Delta^+)$ be the intersection point of $Q_{\xi}$ with the parabola $P_1^{\alpha}$. Finally, if  the invariants of $E_{\alpha}$ are above the parabola $Q_{\xi}$, let $(\mu^+, \Delta^+)$ be the intersection point of $Q_{\xi}$ with the parabola $P_2^{\alpha}$ (see \S \ref{sec-exceptional}). Let $\xi^+$ be the Chern character $(r^+, \mu^+, \Delta^+)$, where $r^+$ is sufficiently high and divisible. The character $\xi^+$ is always a stable Chern character. Let $V$ denote a general point of the moduli space $M(\xi^+)$. 

With this notation, the main theorem of this paper is the following.

\begin{theorem}\label{thm-main-intro}
Let $\xi$ be the Chern character of a stable sheaf such that $M(\xi)$ has Picard rank $2$. Let $U$ be a general point of $M(\xi)$. 
\begin{enumerate}
\item The slope $\mu^+$ is the solution to the higher rank interpolation problem for $U$. 

\item The Brill-Noether divisor $D_V$ spans an extremal ray of $\Eff(M(\xi))$.

\item If the invariants of $E_\alpha$ do not lie on $Q_\xi$, then $D_V$ also spans an extremal ray of $\Mov(M(\xi))$.
\end{enumerate}
\end{theorem}

The proof of Theorem \ref{thm-main-intro} relies on finding an appropriate resolution of $U$. In \S \ref{sec-beilinson}, we will give a resolution of $U$ in terms of a triad of exceptional bundles. The choice of optimal triad is determined by the intersection of $Q_{\xi}$ with the line $\Delta = \frac{1}{2}$,  and the resolution naturally arises from the Beilinson spectral sequence. This resolution allows one to compute the cohomology of $U \otimes V$ effectively. 

We will show that $M(\xi)$ always admits a dominant rational map to a moduli space of appropriately chosen Kronecker modules (see \S \ref{sec-effective}). As long as the invariants of the exceptional bundle $E_{\alpha}$ do not lie on $Q_{\xi}$, this map has positive dimensional fibers. Consequently, the pullback of the ample generator from the moduli space of Kronecker modules gives the edge of both the effective and movable cones. When the invariants of $E_{\alpha}$ lie on $Q_{\xi}$, the corresponding rational map is birational and contracts a divisor. The exceptional locus gives an extremal edge of the effective cone and the pullback of the ample generator on the moduli space of Kronecker modules defines a movable class. 

The curve $\delta(\mu)$ intersects the line $\Delta = \frac{1}{2}$ in a Cantor set $C$ (see \S \ref{sec-transcendental}), which consists precisely of the real numbers lying in no interval $I_\alpha$.  One of the main technical points that we have to address is to show that the parabola $Q_{\xi}$ does not intersect the line $\Delta = \frac{1}{2}$ along $C$ so that there is always a well-defined interval $I_{\alpha}$. In \S \ref{sec-transcendental}, we will prove that every point of the Cantor set which is not an end point of an interval $I_{\alpha}$ is transcendental (see Theorem \ref{cantorTranscendThm}). We will further show that $Q_{\xi}$ cannot pass through an end point of an interval $I_{\alpha}$. Since the intersection of $Q_{\xi}$ with $\Delta = \frac{1}{2}$ is either a rational number or a quadratic irrational, it follows that $Q_{\xi}$ does not intersect the Cantor set. These results rely on remarkable properties of the continued fraction expansions of points of $C$ and deep Diophantine approximation results such as Roth's Theorem \cite{Roth}  and the Schmidt Subspace Theorem \cite{Schmidt}.

The moduli space $M(\xi)$ admits a morphism $j: M(\xi) \rightarrow M^{DUY}(\xi)$ to the Donaldson-Uhlenbeck-Yau moduli space $M^{DUY}(\xi)$ constructed by Jun Li (see \cite{JunLiDonaldson}, \cite{HuybrechtsLehn}). The nef divisor defining the morphism $j$ spans an extremal  ray $R$ of the nef cone. The divisor constructed in Theorem \ref{thm-main-intro} lies on the same side of $R$ as the ample cone. The other extremal edge of $\Eff(M(\xi))$ lies on the opposite side of the ample cone. When the rank of $\xi$ is at least three,  this edge of  $\Eff(M(\xi))$ is determined by duality and  given by  $h^2$-Brill-Noether divisors corresponding to the negative intersection of the parabola $Q_{\xi}$ with the parabolas defining the curve $\delta(\mu)$.   

In the remaining cases, the second extremal ray of  $\Eff(M(\xi))$ has a simple geometric description. 
When the rank of $\xi$ is one or two, the exceptional locus of the morphism $j$ is divisorial and spans this ray.  In rank 2, the locus of sheaves in $M(\xi)$ that fail to be locally free is an effective divisor which is  contracted by the morphism $j$. Similarly, the other edge of the movable cone is given by the nef divisor that defines $j$. This case was discovered earlier by Stromme \cite{Stromme}. When the rank is one, the Donaldson-Uhlenbeck-Yau moduli space coincides with the symmetric product $\PP^{2(n)}$ and $j$ is the Hilbert-Chow morphism.  The other extremal ray of the effective cone is  spanned by the divisor of non-reduced schemes. The other extremal ray of the movable cone is spanned by $j^*(\OO_{\PP^{2(n)}}(1))$ (see \cite{ABCH}). 

When $\xi$ has rank zero and $M(\xi)$ has Picard rank $2$, then we must have $d:= c_1(\xi) \geq 3$, and $M(\xi)$ parameterizes $S$-equivalence classes of pure $1$-dimensional sheaves with Chern character $\xi$.  Let $s: M(\xi) \rightarrow |\OO_{\P^2}(d)|=\P^N$ be the morphism sending a sheaf to its support.  Then $s^* \OO_{\P^N}(1)$ spans the other extremal ray of both the movable and the effective cones \cite{Woolf}.

\subsection*{Organization of the paper}  In \S \ref{sec-prelim}, we will recall basic facts concerning coherent sheaves, exceptional bundles, the classification of stable bundles, the Picard group of the moduli space, and Brill-Noether divisors.  In \S \ref{sec-exceptional}, we will show how to associate an exceptional bundle to a given Chern character $\xi$ and state our theorem on the effective cone precisely. In \S \ref{sec-transcendental}, we will study the number-theoretic properties of points in the Cantor set $C$ and show that every point of $C$ which is not an end point of an interval is transcendental. In \S \ref{sec-beilinson}, we will use the corresponding exceptional bundle to choose an appropriate triad of exceptional bundles.  We then use the Beilinson spectral sequence to determine  an important resolution of the general sheaf $U\in M(\xi)$.  In \S \ref{sec-kronecker}, we define a rational map from the moduli space of sheaves on $\PP^2$ to a moduli space of Kronecker modules. This map arises from the special resolution of the general sheaf.  In \S \ref{sec-effective},  we will determine the effective cone of the moduli space of sheaves.  We close the paper with a sample explicit computation of the effective cone of a moduli space.

\subsection*{Acknowledgments} We are grateful to Daniele Arcara, Arend Bayer, Aaron Bertram, Lawrence Ein, Joe Harris, Brendan Hassett, Emanuele Macr\`{i}  and Yuri Tschinkel for many useful conversations about the geometry of the moduli spaces of sheaves.  We would also like to thank the Vietnam Institute for Advanced Study in Mathematics (VIASM) for providing a wonderful working environment while this work was completed.

\section{Preliminaries}\label{sec-prelim}

In this section, we collect basic facts concerning the classification of stable bundles on $\PP^2$ and the geometry of the Gieseker moduli spaces of semistable sheaves on $\PP^2$. We refer the reader to books by Le  Potier \cite{LePotierLectures} and Huybrechts and Lehn \cite{HuybrechtsLehn} for more details.

\subsection{Stability; slope and discriminant; Euler pairing}  In this paper all sheaves will be coherent. Let $E$ be a coherent sheaf.  The dimension $\dim E$ of $E$ is the dimension of the support $\Supp(E)$.  We say $E$ is \emph{pure of dimension $d$} if $\dim E = d$ and every non-trivial coherent subsheaf $F\subset E$ has $\dim F = d$.  If $E$ has dimension $d$, the Hilbert polynomial $P_E(m) = \chi(E(m))$ of $E$ is of the form $$P_E(m)=\alpha_d \frac{m^d}{d!} + O(m^{d-1}).$$ The \emph{reduced Hilbert polynomial} $p_E$ is defined by $$p_E = \frac{P_E}{\alpha_d}.$$ Then $E$ is \emph{(Gieseker) semistable} (resp. \emph{stable}) if $E$ is pure and for every nontrivial $F\subset E$ we have $p_F\leq p_E$ (resp. $<$), where polynomials are compared for all sufficiently large $m$.  

Fix a Chern character $\xi = (r, \ch_1, \ch_2)$. If the rank $r>0$, then the {\em slope} $\mu$ and the {\em discriminant} $\Delta$ are defined by  $$\mu = \frac{\ch_1}{r}, \ \ \ \mbox{and} \ \ \ \Delta =  \frac{1}{2} \mu^2 - \frac{\ch_2}{r}.$$  The rank, slope and discriminant determine the Chern character. Hence, we can equivalently record the Chern character by $\xi = (r, \mu, \Delta)$. The classification of stable vector bundles on $\PP^2$ is most conveniently expressed in terms of  these invariants, so we will primarily use them instead of $\ch_1$ and $\ch_2$. If $r>0$,  the Riemann-Roch formula reads $$\chi(E) = r (P(\mu)-\Delta),$$ where $P(m)$ is the Hilbert polynomial of $\OO_{\PP^2}$ given by $$P(m)= \frac{1}{2}(m^2 + 3m + 2).$$   Furthermore, $$\mu(E \otimes F) = \mu(E) + \mu(F) \qquad \textrm{and} \qquad \Delta(E \otimes F) = \Delta(E) + \Delta(F).$$  These definitions easily extend to Chern characters in $K(\P^2)\te \R$, and we observe that if $0\neq c\in \R$ then $\xi$ and $c \xi$ have the same slope and discriminant.  Note that semistability  of a sheaf $E$ is equivalent to requiring that for every proper subsheaf $F\subset E$ we have $\mu(F)\leq \mu(E)$, with $\Delta(F) \geq \Delta(E)$ in case of equality.  A character $\xi$ is \emph{(semi)-stable} if it is the Chern character of a (semi)-stable sheaf.

Given two sheaves $E$ and $F$, their Euler characteristic is defined by the formula
$$\chi(E, F) = \sum_{i=0}^2 (-1)^i \ext^i(E,F).$$ When both sheaves have non-zero rank, the Euler characteristic is computed by the Riemann-Roch formula
$$\chi(E,F)= r(E)r(F)\left(P(\mu(F)- \mu(E)) - \Delta(E) - \Delta(F)\right).$$
The derived dual induces a homomorphism $K(\PP^2) \rightarrow K(\PP^2)$. We will write $\xi^*$ for the dual Chern character. The Euler characteristic depends only on Chern characters, so it induces a bilinear pairing $(\xi, \zeta) = \chi(\xi^*, \zeta)$  on $K(\PP^2)\te \R$.  Correspondingly,  $\xi^{\perp}$ denotes the orthogonal complement of $\xi$ in $K(\P^2)\te \R$ with respect to this pairing.

\subsection{Exceptional bundles} A stable vector bundle $E$ on $\PP^2$ is called an {\em exceptional bundle} if $\Ext^1(E,E)=0$. Line bundles $\OO_{\PP^2}(n)$ and the tangent bundle $T_{\PP^2}$ are examples of exceptional bundles. An {\em exceptional slope} $\alpha$ is the slope of an exceptional bundle.  If $\alpha$ is an exceptional slope, there is a unique exceptional bundle $E_\alpha$ of slope $\alpha$.  The rank of $E_\alpha$ is the smallest positive integer $r_{\alpha}$ such that $r_{\alpha} \alpha$ is an integer. The discriminant $\Delta_{\alpha}$ is then given by $$\Delta_{\alpha} = \frac{1}{2} \left( 1 - \frac{1}{r_{\alpha}^2}\right).$$ We also write $\xi_\alpha = \ch(E_\alpha)$.  All exceptional bundles can be constructed starting with line bundles via the process of mutation \cite{DrezetBeilinson}. Exceptional bundles are precisely the stable bundles $E$ on $\PP^2$ with $\Delta(E) < \frac{1}{2}$ \cite[Proposition 16.1.1]{LePotierLectures}. They are rigid and their moduli spaces consist of a single reduced point \cite[Corollary 16.1.5]{LePotierLectures}.  

There is a one-to-one correspondence between the set of exceptional slopes $\F E$ and dyadic integers $\varepsilon: \ZZ\left[\frac{1}{2}\right] \rightarrow \F E,$ defined inductively by $\varepsilon(n) = n$ for $n \in \ZZ$ and $$\varepsilon\left( \frac{2p+1}{2^{q+1}} \right) = \varepsilon \left(\frac{p\vphantom{1}}{2^q}\right) . \varepsilon \left( \frac{p+1}{2^q}\right), $$ where $\alpha.\beta$ is defined by $$\alpha . \beta = \frac{\alpha + \beta}{2} + \frac{\Delta_{\beta} - \Delta_{\alpha}}{3+ \alpha - \beta}.$$ Define the {\em order} of an exceptional slope $\alpha \in \F E$ to be the smallest natural number $q$ such that $\alpha = \varepsilon(\frac{p}{2^q})$. 

\begin{example}
We record the first several exceptional slopes in the interval $[0,\frac{1}{2}]$ together with their orders.  

\renewcommand{\arraystretch}{1.3}

$$\begin{array}{c|ccccccccc} \frac{p}{2^q}& 0 & \frac{1}{16} &\frac{1}{8}& \frac{3}{16} & \frac{1}{4} & \frac{5}{16} & \frac{3}{8} & \frac{7}{16} & \frac{1}{2}   \\ \hline  \varepsilon\left(\frac{p}{2^q}\right) & 0 & \frac{13}{34} & \frac 5{13} & \frac{75}{194} & \frac 25 & \frac{179}{433} & \frac{12}{29} & \frac{70}{169} & \frac 12 \\
\ord\left(\varepsilon\left(\frac{p}{2^q}\right)\right) & 0 & 4 & 3 & 4 & 2 & 4 & 3 & 4 & 1
 \end{array}$$
\end{example}

\subsection{The classification of stable bundles}\label{subsec-classify}  The classification of positive dimensional moduli spaces of stable vector bundles on $\PP^2$ is expressed in terms of a fractal-like curve $\delta$ in the $(\mu, \Delta)$-plane. Let  $$\delta(\mu) = \sup_{\{\alpha \in \F E :|\mu- \alpha|<3\}} (P(-|\mu-\alpha|) - \Delta_{\alpha}).$$
For each exceptional slope $\alpha \in \F E$, there is an interval $I_{\alpha} = ( \alpha-x_{\alpha}, \alpha + x_{\alpha})$ such that the function $\delta(\mu)$ is defined on $I_{\alpha}$ by $$\delta(\mu) = P(-|\mu - \alpha|) - \Delta_{\alpha}, \ \ \mbox{if} \ \ \mu\in I_\alpha,$$ where $$x_{\alpha} = \frac{3 - \sqrt{5 + 8 \Delta_{\alpha}}}{2}.$$ The graph of $\delta(\mu)$ is an increasing concave up parabola on the interval $[\alpha - x_{\alpha}, \alpha]$ and a decreasing concave up parabola on the interval $[\alpha, \alpha + x_{\alpha}]$.  The graph of $\delta$ over $I_\alpha$ is symmetric across the vertical line $\mu = \alpha$ for $\alpha \in \F E$ and is invariant under translation by integers. Furthermore, for every $\alpha \in \F E$, $\delta(\alpha \pm x_{\alpha}) = \frac{1}{2}$. The graph of $\delta$ intersects the line $\Delta = \frac{1}{2}$ in a Cantor set $$C := \R - \bigcup_{\alpha \in \F E} I_{\alpha}.$$ Every rational number $q\in \Q$ lies in some interval $I_\alpha$; equivalently, the Cantor set consists entirely of irrational numbers \cite[Theorem 1]{Drezet}.

\begin{figure}[htbp]
\begin{center}
\includegraphics[scale=0.95]{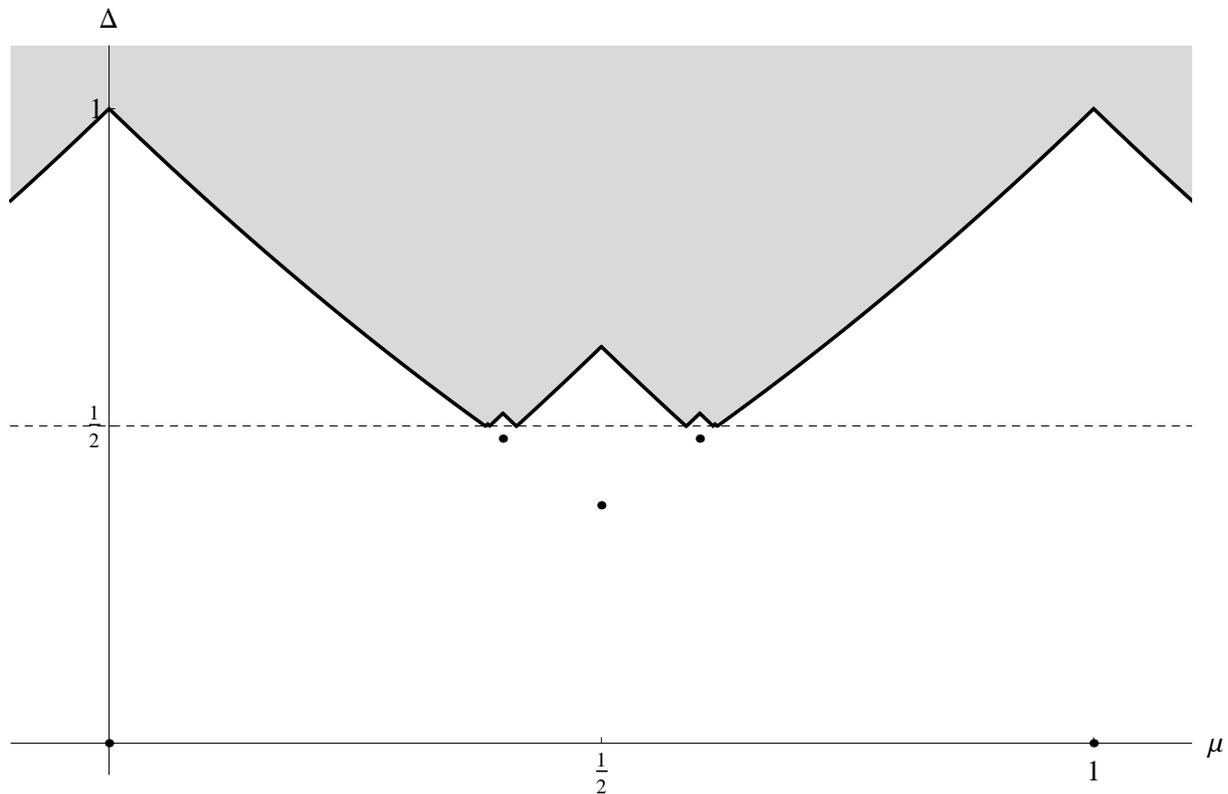}
\end{center}
\caption{The curve $\delta(\mu) =\Delta$ occurring in the classification of stable bundles.  The invariants of the first several exceptional bundles are also displayed.  A Chern character $\xi = (r,\mu,\Delta)$ with $c_1(\xi)\in \Z$ and $\chi(\xi)\in \Z$ has a moduli space $M(\xi)$ of Picard rank $2$ if and only if the point $(\mu,\Delta)$ lies in the shaded region above the curve.  If  $(\mu,\Delta)$ lies on the curve, the moduli space has Picard rank $1$.}
\label{figure-deltaCurve}
\end{figure}

The fundamental theorem on the existence of moduli spaces of semistable sheaves due to Dr\'{e}zet and Le Potier is the following.

\begin{theorem}[\cite{DLP}, \cite{LePotierLectures}]\label{stableThm}
Let $\xi = (r,\mu,\Delta)$ be a Chern character of positive integer rank.  There exists a positive dimensional moduli space of semistable sheaves $M(\xi)$ with Chern character $\xi$ if and only  if $c_1= r \mu \in \ZZ$, $\chi = r(P(\mu) - \Delta) \in \ZZ$ and $\Delta \geq \delta(\mu)$. In this case, $M(\xi)$ is a normal, irreducible, factorial projective variety of dimension $r^2(2 \Delta -1) + 1$.
\end{theorem}

\emph{Height zero} characters $\xi$ are characters with invariants lying on the curve $\Delta = \delta(\mu)$.

Moreover, $M(\xi)$ is a Mori dream space. In particular, the effective cone is closed.  This has been proved in \cite[\S 3]{Woolf}, but only explicitly stated for the rank zero case. We briefly sketch the proof  in the higher rank case and refer to \cite{Woolf} for additional details.

\begin{theorem}
Let $\xi$ be the Chern character of a stable coherent sheaf. Then the moduli space $M(\xi)$ is a Mori dream space. 
\end{theorem} 

\begin{proof}
See \cite{Woolf} for the rank zero case.  We assume $r(\xi)>0$.  By \cite{BCHM}, a log Fano variety is a Mori dream space. Consequently, it suffices to show that $M(\xi)$ is a log Fano variety. By \cite[Theorems 8.2.8 and 8.3.3]{HuybrechtsLehn} the anti-canonical bundle of $M(\xi)$ is nef. By the proof of \cite[Proposition 3.3]{Woolf}, $M(\xi)$ has canonical singularities. Briefly, $M(\xi)$ is the good quotient of a smooth variety \cite[\S 14.5]{LePotierLectures}.  Consequently, $M(\xi)$ has rational singularities \cite{Boutot}. Since $M(\xi)$ is Gorenstein and every rational Gorenstein singularity is canonical \cite[Theorem 11.1]{Kollar}, we conclude that $M(\xi)$ has canonical singularities. Hence, the pair $(M(\xi), \epsilon E)$ is a klt-pair for every effective divisor $E$ and sufficiently small $\epsilon >0$. It will become clear in \S\ref{sec-effective} that there are effective divisors $E$ such that $-K_{M(\xi)} - \epsilon E$ is ample for all  sufficiently small $\epsilon > 0$.  We conclude that $M(\xi)$ is a log Fano variety, hence a Mori dream space. 
\end{proof}

\subsection{Picard group of the moduli space} A  theorem of Dr\'{e}zet determines the Picard group of $M(\xi)$.

\begin{theorem}[Dr\'{e}zet]\label{PicThm}
If $\Delta > \delta(\mu)$, then the Picard group of $M(\xi)$ is a free abelian group on two generators. Furthermore, in this case, $\Pic(M(\xi))\te \R$ is naturally identified with $\xi^{\perp}$. If $\Delta = \delta(\mu)$, then the Picard group of $M(\xi)$ is an infinite cyclic group.  (See Figure \ref{figure-deltaCurve}.)
\end{theorem}

In $M(\xi)$, linear equivalence and numerical equivalence coincide and the N\'{e}ron-Severi space $\NS(M(\xi)) = \Pic(M(\xi)) \otimes \R$. When the Picard rank of a projective variety is one, then the ample, effective and movable cones coincide and are equal to the half-space spanned by any effective divisor. Therefore, for the rest of the paper, we will always assume that $\xi$ is a Chern character satisfying $\Delta > \delta(\mu)$. In this case, the N\'{e}ron-Severi space is a two-dimensional vector space. Hence,  the cones  $\Mov(M(\xi))$ and $\Eff(M(\xi))$ are determined by specifying their two extremal rays. 

At times it will be necessary to consider the open subvariety $M^s(\xi)$ of $M(\xi)$ parameterizing stable sheaves. The space $M^s(\xi)$ is a coarse moduli space for flat families of stable sheaves with Chern character $\xi$. On the other hand, the identification of $S$-equivalence classes of sheaves prevents $M(\xi)$ from being a coarse moduli space unless $\xi$ is a primitive character.  When the Picard rank of $M(\xi)$ is at least $2$, the complement of $M^s(\xi)$ has codimension at least $2$ by \cite[Lemma 18.3.1]{LePotierLectures}.  Thus when studying questions about divisors, there is typically no harm in working on $M^s(\xi)$ instead of $M(\xi)$ as necessary.

Most of the effective divisors considered in this paper will be certain Brill-Noether divisors.  It will be important to understand the class of a Brill-Noether divisor in terms of the isomorphism $\xi^\perp \cong \NS(M(\xi))$.  We first recall the construction of this isomorphism more explicitly.  We summarize the discussion in \cite[Chapter 18]{LePotierLectures}.

Let $\F U/S$ be a flat family of semistable sheaves of Chern character $\xi$ on $\P^2$ over a smooth variety $S$, and let $p:S\times \P^2\to S$ and $q:S\times \P^2\to \P^2$ be the two projections.  A homomorphism $\lambda_{\F U}:K(\P^2)\to \Pic(S)$ is described as the composition
$$K(\P^2) \xrightarrow{q^*} K^0(S\times \P^2) \xrightarrow{\cdot[\F U]}K^0(S \times \P^2) \xrightarrow{-p_!}K^0(S)\xrightarrow{\det} \Pic S,$$ where $p_! = \sum (-1)^i R^i p_\ast.$ If $\F U$ is replaced by a family $\F U \te p^\ast L$ for some line bundle $L$ on $S$, then the moduli map $S\to M(\xi)$ does not change, but $$\lambda_{\F U \te p^*L}(\zeta) = \lambda_{\F U}(\zeta) \te L^{\te- (\xi,\zeta)}.$$  For an integral class $\zeta\in \xi^\perp$, there is a line bundle denoted by $\lambda_M(\zeta)$ on $M(\xi)$ with the property that if $\F U/S$ is a flat family of semistable sheaves with Chern character $\xi$ and $f:S\to M(\xi)$ is the moduli map, then $$\lambda_{\F U}(\zeta) = f^\ast \lambda_M(\zeta).$$ This induces a linear map $\lambda_M:\xi^\perp\to \NS(M(\xi))$ which is an isomorphism when $\Delta>\delta (\mu)$.  

\begin{warning}
We have chosen to normalize $\lambda_M$ by using $-p_!$ instead of $p_!$, as is considered in \cite{HuybrechtsLehn} and \cite{LePotierLectures}.  We will discuss the reason for this change below in Remark \ref{minusRemark}.
\end{warning}

When they exist, Brill-Noether divisors provide explicit sections of some of these line bundles.

\begin{proposition}\label{BNProp}
Let $\xi$ be a positive-rank Chern character such that $M(\xi)$ is of Picard rank $2$.  Suppose $V$ is a stable vector bundle of Chern character $\zeta$ such that the general $U\in M(\xi)$ is cohomologically orthogonal to $V$.  Endow the locus $$D_V =\overline{ \{ U \in M^s(\xi):H^1(U\te V)\neq 0\}}\subset M(\xi)$$ with the natural determinantal scheme structure.
\begin{enumerate}
\item $D_V$ is an effective divisor.
\item If $\mu(U\te V)>-3$, then $\OO_{M(\xi)}(D_V) \cong  \lambda_M(\zeta).$

\item If $\mu(U\te V)<0$, then $\OO_{M(\xi)}(D_V)\cong \lambda_M(\zeta)^* = \lambda_M(-\zeta).$
\end{enumerate}
\end{proposition}
\begin{proof}
We first recall how to give $D_V$ a determinantal scheme structure; the fact that it is a divisor will also follow.  Let $\F U/S$ be a complete flat family of stable sheaves with Chern character $\xi$. We will describe a determinantal scheme structure on the locus $D_{V,S}$ of $s\in S$ such that $U_s \te V\in D_V$.  If $M^s(\xi)$ admits a universal family, this gives the correct scheme structure on $D_V$. Otherwise, one can work with the universal family on the stack $\F M^s(\xi)$ and then take the image in $M^s(\xi)$ (See \cite[Theorem 4.16 and Example 8.7]{Alper}).  

Suppose $\mu(U\te V)>-3$.  Let $d\gg 0$, and fix some $s\in S$. We first describe $D_{V,S}$ locally near $s$.   Let $C$ be a general curve of degree $d$, so that $C$ avoids any singularities of $U_s$.  Replacing $S$ by a smaller open subset $S'$, we may assume $C$ avoids the singularities of $U_s$ for all $s\in S'$.  Consider the exact sequence $$0\to \F U \te q^*V \to \F U \te q^*(V(d)) \to \F U \te q^*(V(d)|_C)\to 0$$ of sheaves on $S'\times \P^2$.  For a fixed $s\in S'$ we have a sequence $$H^1(U_s \te V(d))\to H^1(U_s\te V(d)|_C))\to H^2(U_s \te V).$$ The first cohomology group here vanishes since $d\gg 0$ and the moduli space of semistable sheaves $M(\xi \te \zeta)$ is bounded.  The last group vanishes by stability since $\mu(U\te V)>-3$.  Thus the middle group vanishes, and we conclude $R^1p_*(\F U\te q^*(V(d)|_C))=0$.  We also have $p_*(\F U\te q^*V)=0$ since the general $U_s$ is cohomologically orthogonal to $V$. We then obtain an exact sequence  $$0\to  p_*(\F U \te q^*(V(d))) \xrightarrow{\phi} p_*(\F U \te q^*(V(d)|_C)) \to R^1p_* (\F U\te q^*V)\to 0$$ on $S'$.  The sheaves $A=p_*(\F U\te q^*(V(d)))$ and $B=p_*(\F U\te q^*(V(d)|_C))$ are vector bundles of the same rank since $d$ is sufficiently large and $(\xi,\zeta)=0$.   The divisor $D_{V,S'}$ is the zero locus of the section $\det \phi$ of the line bundle $\sHom(\det A,\det B) \cong \det R^1p_*(\F U\te q^*V)$, so it has a scheme structure.

In fact, the scheme structure on $D_{V,S'}$ described locally in the previous paragraph is independent of the choice of $d$ and $C$, and the construction globalizes to give a scheme structure on the divisor $D_{V,S}$.  Precisely, over an open set $S'$ which is sufficiently small that $A$ and $B$ are isomorphic trivial bundles, the Fitting ideal sheaf $\Fitt_0(R^1p_*(\F U \te q^*V))\subset \OO_S$ is generated by $\det \phi$, viewed as a section of $\OO_{S'}$ via the trivializations of $A$ and $B$ (See \cite[Chapter 20]{Eisenbud2}).  This Fitting ideal is therefore the ideal sheaf of $D_{V,S}$, and the scheme structure on $D_{V,S}$ is determined. 

We can further identify $$\Fitt_0(R^1p_*(\F U\te q^*V))^*\cong \det R^1 p_*(\F U\te q^*V).$$ Indeed, on an open set $S'$ where $A$ and $B$ are locally trivial, there is an induced trivialization of the line bundle $\det R^1 p_*(\F U\te q^*V)$ and a local section $\det \phi$, viewed as a section of $\OO_{S'}$.  The collection of all such local sections forms an effective Cartier divisor by the existence of the Fitting ideal.  This induces a global section of $\det R^1p_*(\F U\te q^*V)$ that yields an isomorphism $\det R^1 p_*(\F U\te q^*V)\cong \OO_S(D_{V,S})$.

Finally,
using $V$ to compute $\lambda_{\F U}(\xi')$, we find $$\lambda_{\F U}(\zeta) = \det (-p_![\F U \te q^*V])=\det ([R^1p_*(\F U \te q^*V)]) = \OO_S(D_{V,S}),$$ with $R^2p_* (\F U\te q^*V)$ vanishing by stability and the slope assumption. We conclude the corresponding result holds for the moduli space as well.

When $\mu(U\te V) <0$, we instead consider an exact sequence $$0 \to \F U \te q^*(V(-d)) \to \F U\te q^*V \to \F U\te q^*V|_C\to 0$$ for $d\gg 0$ and perform a similar calculation.  We find $\OO_S(D_{V,S})\cong \det R^2p_* (\F U\te q^*V)$ and compute $\lambda_{\F U}(\zeta)$ as before.
\end{proof}

\begin{remark}
In the proposition, we described the divisor $D_V$ as the locus of all $U$ where $H^1(U\te V)$ jumps in rank.  As $\chi(U\te V)=0$, if $\mu(U\te V)\geq 0$ then this locus  can instead be described as the locus where $H^0(U\te V)$ jumps rank.  Similarly, if $\mu(U\te V)\leq -3$, it can be described as the locus where $H^2(U\te V)$ jumps rank.  These descriptions are perhaps preferable, in light of the fact that the divisor class of $D_V$ exhibits qualitatively different behavior in each case despite the uniform description in terms of $H^1$.  

It is also worth pointing out that $U$ and $V$ are never cohomologically orthogonal if $-3<\mu(U\te V)<0$.  For if they are, then $\lambda_M(\zeta)=\lambda_M(-\zeta)$.  As $\lambda_M$ is an isomorphism, this is impossible.  The hypothesis that the Picard rank of $M(\xi)$ is $2$ is crucial here.  If $\xi$ is a stable Chern character such that $\chi(\xi)=0$ and $-3<\mu(\xi)<0$, then any $U\in M(\xi)$ is cohomologically orthogonal to $\OO_{\P^2}$.  It follows that $M(\xi)$ has Picard rank at most $1$.  (This is also clear from Theorem \ref{PicThm}.)
\end{remark}

\subsection{Natural line bundles on the moduli space}\label{ssec-primary}
Following Huybrechts and Lehn \cite{HuybrechtsLehn}, we define two Chern characters $\zeta_0,\zeta_1$ depending on the character $\xi$ of the moduli space by the formula $$\zeta_i = r(\xi) h^i -(\xi,h^i) h^2,$$ where $h$ is the Chern character of the structure sheaf $\OO_H$ of a line $H\subset \P^2$.  (These differ from the definition in \cite{HuybrechtsLehn} by a sign due to our convention for $\lambda_M$.) Since $(\xi,h^2) = r(\xi)$, we see that $\zeta_i\in \xi^\perp$.  More explicitly, we compute the Chern characters $(\ch_0,\ch_1,\ch_2)$ of $\zeta_i$ as 
\begin{align*}
\zeta_0 &= (r(\xi),0,-\chi(\xi))\\
\zeta_1 &= (0,r(\xi),-\frac{3}{2}r(\xi)-c_1(\xi))
\end{align*}
Line bundles $\F L_i$ are then defined by $\F L_i = \lambda_M(\zeta_i)$.

The line bundle $\F L_1$ gives the morphism $M(\xi)\to M^{DUY}(\xi)$ to the Donaldson-Uhlenbeck-Yau compactification by slope-semistable sheaves.  It spans an edge of the nef cone of $M(\xi)$.  For $n\gg 0$, the line bundle $\F L_0 \te \F L_1^n$ is ample \cite{HuybrechtsLehn}.

We view the Chern character $\zeta_1$ corresponding to $\F L_1$ as spanning a line in $\xi^\perp$ which splits this plane into two open halves, depending on the rank of a Chern character.  We call the open half-plane of characters with positive rank the \emph{primary} half, and the open half-plane of characters with negative rank the \emph{secondary} half.  We also split $\NS(M(\xi))$ into primary and secondary halves according to the isomorphism $\NS( M(\xi))\cong \xi^\perp$.  The ample cone of $M(\xi)$ lies in the primary half since $\zeta_0$ lies in the primary half.  

Given a closed full-dimensional  cone in $\NS(M(\xi))$ which contains $\F L_1$ and is not equal to either half-plane, at least one of its edges lies in one of the half-planes; we call the edge \emph{primary} or \emph{secondary} accordingly.  The other edge (which possibly coincides with the ray spanned by $\F L_1$) is given the opposite name, so that the cone has a single primary and a single secondary edge.  When the edge of a cone is given by a Brill-Noether divisor, it is simple to determine whether the edge is primary or secondary.
The next result follows immediately from Proposition \ref{BNProp}.

\begin{corollary}\label{primaryCor}
Keep the notation and hypotheses from Proposition \ref{BNProp}.
\begin{enumerate}
\item If $\mu(U\te V)>-3$, then $[D_V]$ lies in the primary half of $\NS(M(\xi))$.

\item If $\mu(U\te V)<0$, then $[D_V]$ lies in secondary half of $\NS(M(\xi))$.
\end{enumerate}
\end{corollary}

\begin{remark}\label{minusRemark}
Without our sign convention on $\lambda_M$, the $h^0$-Brill-Noether divisor associated to a bundle $V$ of character $\zeta$ with $\mu(U\te V)>-3$ corresponds to the Chern character $-\zeta$ of negative rank.  The primary half of $\xi^\perp$ must be taken to consist of the negative rank characters, and the ample cone of $M(\xi)$ corresponds to certain characters of negative rank.  Furthermore, the classes $\zeta_0$ and $\zeta_1$ would have negative rank and negative $\ch_1$, respectively.  As we will be thinking of $\NS(M(\xi))$ primarily in terms of $h^0$-Brill-Noether divisors, it is advantageous to renormalize $\lambda_M$ and alleviate all these sign issues.
\end{remark}

\section{The corresponding exceptional bundle}\label{sec-exceptional}

For the rest of the paper, $\xi$ will denote a Chern character of positive rank such that the moduli space $M(\xi)$ of semistable sheaves with Chern character $\xi$ is of Picard rank $2$.  We will only lift this restriction in \S \ref{ssec-rank0}, where we discuss the case of torsion sheaves.

In this section, we introduce an exceptional bundle which controls the effective  cone of divisors on $M(\xi)$. We also discuss how the effective cone is determined by this bundle.  We focus on describing the primary edge of the effective cone, as defined in Section \ref{ssec-primary}.  When the rank is at least $3$, the secondary edge of the cone will be easily determined by a duality argument. For smaller ranks, the secondary edge exhibits special behavior.

Recall that $\xi^\perp$ is a linear space which is naturally identified with $\NS(M(\xi))$.  The rank of an orthogonal class $\zeta\in \xi^\perp$ is relatively unimportant, so it is convenient to work with slopes and discriminants instead of Chern characters. Hence, we  introduce the parabola $$Q_\xi = \{(\mu,\Delta): (1,\mu,\Delta) \textrm{ lies in } \xi^\perp\} \subset \R^2$$ as a subset of the $(\mu,\Delta)$-plane.

The Riemann-Roch formula shows that $Q_\xi$ is the concave up parabola with equation $$Q_\xi:P(\mu(\xi)+\mu)-\Delta(\xi)=\Delta.$$ where $P(x) = \frac{1}{2}(x^2+3x+2)$ is the Hilbert polynomial of $\OO_{\P^2}$.  Consider the \emph{reference parabola} $Q_{\xi_0}$ corresponding to the Chern character $\xi_0$ of $\OO_{\P^2}$.  It is defined by the equation $$Q_{\xi_0}: P(\mu)=\Delta,$$ so $Q_\xi$ is obtained from $Q_{\xi_0}$ by translation by the vector $(-\mu(\xi),-\Delta(\xi)).$ The hypothesis that $M(\xi)$ is positive-dimensional gives $\Delta(\xi)>\frac{1}{2}$.  The minimum value of $P(x)$ occurs at $P(-\frac{3}{2})=-\frac{1}{8}$, so $Q_\xi$ passes below the line $\Delta=0$.

Observe that every parabolic arc in the curve $\delta(\mu)=\Delta$ is a portion of some parabola $Q_{\xi_\alpha}$, where $\xi_\alpha$ denotes the Chern character of an exceptional slope $\alpha\in \F E$.  If $\xi$, $\zeta$ are any two linearly independent (rational) Chern characters of positive rank, then the intersection $Q_\xi \cap Q_{\zeta}$ consists of a single point $\{(\mu,\Delta)\}$ with rational coordinates, corresponding to the intersection of planes $\xi^\perp \cap {\zeta}^\perp$.

\begin{theorem}\label{existenceOfSlope}
The parabola $Q_\xi$ intersects the line $\Delta = \frac{1}{2}$ at two points.  If $\mu_{0}\in \R$ is the larger of the two slopes such that $(\mu_{0},\frac{1}{2})\in Q_\xi$, then there is a (unique) exceptional slope $\alpha\in \F E$ such that $\mu_{0}\in I_\alpha$.
\end{theorem}
\begin{proof}
We sketch the proof, deferring the most difficult part to the next section.  We first explicitly calculate $$\mu_{0} = \frac{ - 3 - 2\mu(\xi)+\sqrt{5+8\Delta(\xi)}}{2}$$ from the formula for $Q_\xi$.  Then either $\mu_{0}$ is rational (in case $5+8\Delta(\xi)\in \Q$ is a square) or it is a quadratic irrational.  If $\mu_{0}$ is rational, then $\mu_{0}\in I_\alpha$ for some $\alpha$ by \cite[Theorem 1]{Drezet}.

Suppose $\mu_{0}$ is a quadratic irrational, and that $\mu_{0}$ is not in any $I_\alpha$, so that it lies in the generalized Cantor set $C = \R \sm \bigcup_\alpha I_\alpha$.  We show in the next section that any $x\in C$ which is \emph{not} an endpoint of some interval $I_\alpha$ is  transcendental.  Hence it suffices to show that $\mu_0$ cannot be an endpoint. Without loss of generality, assume that $\mu_0 = \alpha + x_\alpha$ for some $\alpha\in \F E$.  The point $(\alpha+ x_\alpha,\frac{1}{2})$ is the endpoint of one of the parabolic arcs in the curve $\delta(\mu) = \Delta$, so it lies on a parabola $Q_{\xi_\beta}$ for some exceptional slope $\beta$.   Then $Q_{\xi} \cap Q_{\xi_\beta} = \{(\mu_{0},\frac{1}{2})\}$ is a single point with rational coordinates.  Therefore, $\mu_0$ is rational, a contradiction.
\end{proof}

\begin{definition}\label{def-exceptional-slope}
The exceptional slope $\alpha$ in Theorem \ref{existenceOfSlope} is the \emph{(primary) corresponding exceptional slope} to $\xi$.
\end{definition}

It is now easy to describe the ray spanning the primary edge of the effective cone.   The behavior of the effective  cone depends on the sign of the pairing $(\xi,\xi_\alpha) = \chi(\xi \te \xi_{\alpha})$.  We keep the identification $\NS(M(\xi))\cong \xi^\perp$ in mind, and recall that the primary half of this space corresponds to Chern characters of positive rank.

\begin{enumerate}
\item If $(\xi,\xi_\alpha)>0$, the primary edge of the effective cone is spanned by a positive rank Chern character in $\xi^\perp \cap (\xi_{-\alpha})^\perp$.

\item If $(\xi,\xi_\alpha)=0$, the primary edge of the effective cone is spanned by the orthogonal Chern character $\xi_\alpha$.  
\item If $(\xi,\xi_\alpha)<0$, the primary edge of the effective cone is spanned by a positive rank Chern character in  $\xi^\perp \cap (\xi_{-\alpha-3})^\perp$.
\end{enumerate}

Note that the previous intersections are also easily computed in the $(\mu,\Delta)$-plane as intersections of parabolas. For example, in the first case, the primary edge of the effective cone corresponds to the point $Q_\xi \cap Q_{\xi_{-\alpha}}$.

\begin{definition}
The \emph{(primary) corresponding orthogonal invariants} $(\mu^+,\Delta^+)$ to $\xi$ are defined by 
\begin{enumerate}
\item $\{(\mu^+,\Delta^+)\} = Q_\xi \cap Q_{\xi_{-\alpha}}$ if $(\xi,\xi_\alpha)>0$,
\item $(\mu^+,\Delta^+)=(\alpha,\Delta_\alpha)$ if $(\xi,\xi_\alpha)=0$, and 
\item $\{(\mu^+,\Delta^+)\} = Q_\xi \cap Q_{\xi_{-\alpha-3}}$ if $(\xi,\xi_\alpha)< 0$.
\end{enumerate}
A \emph{(primary) orthogonal Chern character $\xi^+$ to $\xi$} is defined to be any character $\xi^+=(r^+,\mu^+,\Delta^+)$ where $r^+$ is sufficiently large and divisible.
\end{definition}

We will usually drop the parenthesized word primary when discussing the previous constructions, as we will only briefly discuss the secondary edge of the effective cone in \S\ref{ssec-secondary}. With this notation, we summarize the preceding discussion with the most concise statement of the theorem we will eventually prove.

\begin{theorem}
In terms of the isomorphism $\NS(M(\xi))\cong \xi^\perp$, the primary edge of the effective cone of $M(\xi)$ is spanned by any primary orthogonal character $\xi^+$.  
\end{theorem}

It is useful to compare the corresponding orthogonal invariants with the classification of stable vector bundles.

\begin{remark}
It is useful to sketch the intersections defining the corresponding orthogonal invariants in the $(\mu,\Delta)$-plane in the three separate cases and to compare them with the classification of stable vector bundles.  First, in Figure \ref{figure-triangle} we display the important features of the portion of the curve $\delta(\mu) = \Delta$ lying over the interval $I_\alpha$ of slopes.  The curve consists of parabolic arcs $Q_{\xi_{-\alpha}}$ and $Q_{\xi_{-\alpha-3}}$ meeting at the vertical line $\mu = \alpha$.  

\begin{figure}[htbp]
\begin{center}
\input{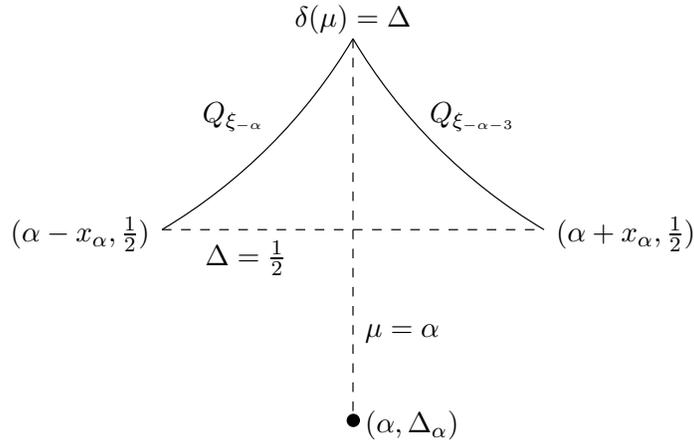}
\end{center}
\caption{Schematic diagram of the portion of the curve $\delta(\mu) = \Delta$ over the interval $I_\alpha$ in the $(\mu,\Delta)$-plane.}
\label{figure-triangle}
\end{figure}

Next, in Figure \ref{figure-arrangements} we overlay the sketch of the curve $\delta(\mu) = \Delta$ with the parabola $Q_\xi$ in several possible arrangements. 

\begin{enumerate}\item We have $(\xi,\xi_\alpha)>0$ and $(\mu^+,\Delta^+)$ lies on the $\delta$-curve.  

\item[($1'$)] Here we still have $(\xi,\xi_\alpha)>0$, but $(\mu^+,\Delta^+)$ does \emph{not} lie on the $\delta$-curve.  Correspondingly, $\mu^+>\alpha$.

\item In this case $Q_\xi$ passes through the Chern character $\xi_\alpha$, and $(\xi,\xi_\alpha)=0$.  The effective cone corresponds to $\xi_\alpha$.

\item Finally, when $(\xi,\xi_\alpha)<0$, the point $(\mu^+,\Delta^+)$ always lies on the $\delta$-curve.
\end{enumerate}
When we are in case (1), (2), or (3), we observe that $\mu^+$ is the minimum possible slope of a \emph{stable} Chern character $\zeta$ orthogonal to $\xi$ such that $\mu(\xi\te \zeta)\geq 0$.  Conversely, in case ($1'$), there exist stable Chern characters $\zeta$ orthogonal to $\xi$ with $\mu(\zeta)<\mu^+$ such that $\mu(\xi\te \zeta)\geq0$, and such Chern characters span noneffective rays in $\Pic(M(\xi))$.

\begin{figure}[htbp]
\begin{center}
\hspace{-.2in}\input{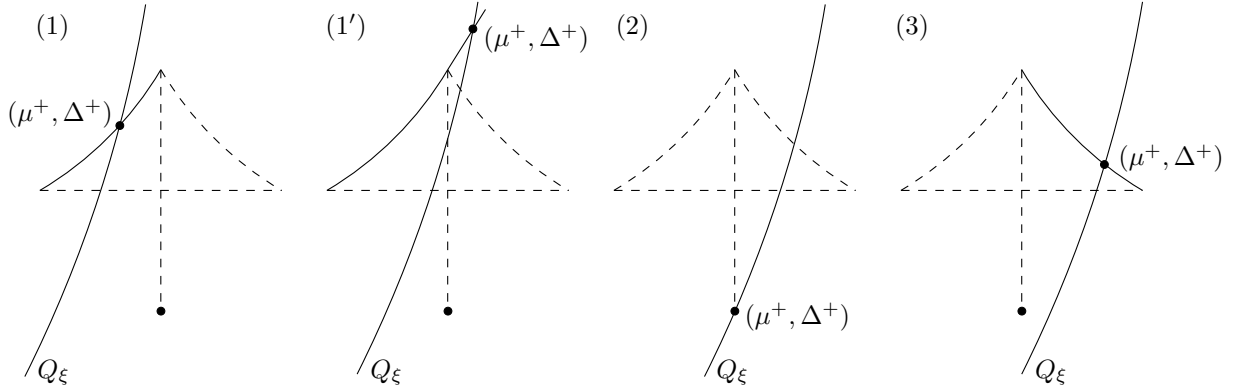}
\end{center}
\caption{Possible relative positions of the curve $\delta(\mu)=\Delta$ and the parabola $Q_\xi$.}
\label{figure-arrangements}
\end{figure}
\end{remark}

When $(\xi,\xi_\alpha)<0$, it is clear by the Intermediate Value Theorem that $Q_\xi$ intersects the parabola $Q_{\xi_{-\alpha-3}}$ as in case (3) above.  To verify that the schematic diagram is correct when $(\xi,\xi_\alpha)>0$, we must perform a small check.

\begin{lemma}\label{stupidInequalityLemma}
Suppose $(\xi,\xi_\alpha)>0$.  Then $\mu^+>\alpha-x_\alpha$.
\end{lemma}
\begin{proof}
Let $\mu_{0}\in \R$ be the larger of the two slopes such that $(\mu_{0},\frac{1}{2})\in Q_\xi$.  Then $\mu_0\in I_\alpha$.  Let $(\mu_1,\Delta_1)$ and $(\mu_2,\Delta_2)$ be the vertices of $Q_\xi$ and $Q_{\xi_{-\alpha}}$, respectively.  We compute $$ \mu_1= -\frac{3}{2}-\mu(\xi)\qquad \Delta_1= -\frac{1}{8}-\Delta(\xi) \qquad  \mu_2 = -\frac{3}{2}+\alpha \qquad \Delta_2= -\frac{1}{8}-\Delta_\alpha.$$ 

Let $U$ be a slope-stable bundle of character $\xi$.
Since $(\xi,\xi_\alpha)>0$, there is a nonzero homomorphism $E_{-\alpha}\to U$ by stability; see Lemma \ref{lem-estimates} (1) for details.  Stability gives $-\alpha<\mu(\xi)$, from which we conclude $\mu_2>\mu_1$.  Since $\Delta(\xi)>\frac{1}{2}$, we also have $\Delta_2>\Delta_1$.  

The point $p=(\alpha-x_\alpha,\frac{1}{2})$ lies on $Q_{\xi_{-\alpha}}$ and above $Q_\xi$. The vertex of $Q_{\xi_{-\alpha}}$ lies to the left of $p$.  Since $Q_\xi$ and $Q_{\xi_{-\alpha}}$ are translates of one another, they intersect in a single point, and the Intermediate Value Theorem shows the vertex of $Q_{\xi_{-\alpha}}$ lies above $Q_{\xi}$.  An elementary calculation implies that $\mu^+>\alpha-x_\alpha,$ as in Figure \ref{figure-arrangements}.
\end{proof}

The following result is now immediate from Theorem \ref{stableThm}.

\begin{proposition}\label{stableCharacterProp}
The Chern character $\xi^+$ is semistable.
\end{proposition}

\begin{remark}
Case ($1'$) of the previous remark never  arises for Chern characters $\xi$ of rank $1$ (or rank $0$, in a suitable sense that will become clear in \S \ref{ssec-rank0}).  In particular, for the Hilbert scheme of points, the primary edge of the effective cone always corresponds to the stable orthogonal Chern character of smallest positive slope \cite{HuizengaPaper2}.  It follows from results in \cite{HuizengaPaper2} and \cite{CoskunHuizenga} that if $\xi = (1,0,-n)$ is the Chern character of an ideal sheaf of $n$ points and $\zeta$ is a stable Chern character with $\mu(\zeta)>0$ such that $(\xi,\zeta)=0$, then the image of the rational map \begin{eqnarray*}M(\xi)\times M(\zeta) &\dashrightarrow & M(\xi \te \zeta)\\ (U,V) & \mapsto &U\te V \end{eqnarray*} is not contained entirely in the $\Theta$-divisor of sheaves with a section.

The natural analog of this statement is false in higher rank.  For an example, consider the moduli space $M(\xi) = M(r,\mu,\Delta)=M(2,0,\frac{11}{2})$. The associated exceptional slope is $\alpha=2$, and $(\xi,\xi_\alpha)=1$.  Furthermore, $\mu^+ = \frac{21}{10} > \alpha$, so we are in case ($1'$) of the previous remark.  If $(\mu_0,\Delta_0)$ is the point of intersection of $Q_\xi$ with the $\delta$-curve, then for sufficiently divisible $r_0$ the Chern character $(r_0,\mu_0,\Delta_0)$ is stable and orthogonal to $\xi$, but corresponds to a noneffective divisor class on $M(\xi)$.  Thus the image of the tensor product map $M(\xi)\times M(\zeta)\dashrightarrow M(\xi\te \zeta)$ must land in the $\Theta$-divisor, for otherwise we would obtain an effective divisor of class $\zeta$ on $M(\xi)$.
\end{remark}

\section{Transcendentality of the generalized Cantor set}\label{sec-transcendental}

Let $C = \R \sm \bigcup_{\alpha \in \F E} I_\alpha$ be the generalized Cantor set, where $\F E$ is the set of exceptional slopes. In this section, we prove the following number-theoretic fact about points of $C$.

\begin{theorem}\label{cantorTranscendThm}
Let $x$ be a point of $C$. Then either $x$ is an endpoint of an interval $I_\alpha$ and is a quadratic irrational number, or $x$ is transcendental.
\end{theorem}

Theorem \ref{cantorTranscendThm} completes the proof of Theorem \ref{existenceOfSlope}, establishing the existence of the corresponding exceptional slope. Before giving the details, we will outline the roadmap of the proof. The reader interested solely in geometry can safely skip to the next section, but we have written the argument for readers with minimal number theory background in mind.

The proof of the theorem relies on remarkable number-theoretic properties of exceptional slopes.  Since the set $C$ is stable under translation by integers $x\mapsto x+k$ and inversion $x\mapsto -x$, we will restrict ourselves to $C\cap (0,\frac{1}{2})$. Every rational number $q$ has an even length $q^e$ and an odd length $q^o$ continued fraction expansion. The even continued fraction expansion $\beta^e$ of an exceptional slope $0< \beta < \frac{1}{2}$ is a palindrome made of a string of ones and twos beginning and ending with a two. Furthermore, given two exceptional slopes $$\alpha = \varepsilon \left( \frac{p-1}{2^q}\right)\qquad \textrm{and}\qquad \beta = \varepsilon \left( \frac{p+1}{2^q} \right),$$ the even partial fraction expansion of $\gamma = \alpha . \beta = \varepsilon\left( \frac{p}{2^q}\right)$ is obtained by   concatenating the odd partial fraction expansion of $\beta$ with the even partial fraction expansion of $\alpha$ with a two in between (see \S \ref{subsec-continued}), i.e. $$\gamma^e = \beta^o 2 \alpha^e.$$ Using this fact, we give an explicit algorithm for computing the continued fraction expansion of any point of the Cantor set $C$ to arbitrary length.  The crux of the matter is to show that given $x \in C$ which is not an endpoint of $I_{\alpha}$ for some $\alpha$, the continued fraction expansion of $x$ is not eventually periodic (Theorem \ref{thm-cantor-notquad}). A celebrated theorem of Euler and Lagrange \cite{Davenport} asserts that a number is a quadratic irrational if and only if its continued fraction expansion is eventually periodic. This suffices to prove that $x$ is not a quadratic irrational. To prove the stronger fact that $x$ is transcendental, we need a theorem of Adamczewski and Bugeaud.

\begin{theorem}[Adamczewski-Bugeaud, Theorem 2.1 \cite{AdB}]\label{thm-AdB}
If the continued fraction expansion of a number $x$ begins in arbitrarily long palindromes, then $x$ is either quadratic irrational or transcendental.
\end{theorem}
If $x\in C$ is not an endpoint of some interval $I_\alpha$, we will show that its continued fraction expansion is not eventually periodic and  begins in arbitrarily long palindromes.  We conclude that  $x$ is transcendental. 

We now turn to the details. At times,  it will  be convenient to  consider certain portions of the Cantor set.  Let  $C' = C\cap (-1,1)$ and $C'' = C\cap (0,\frac{1}{2})$ for brevity.  By symmetry, it suffices to prove the theorem for $C''$.  Similarly, we define restricted intervals of exceptional slopes $\F E' = \F E \cap (-1,1)$ and $\F E'' = \F E\cap (0,\frac{1}{2})$.

\subsection{The Left-Right sequences of points in $C'$}

Recall that the order of an exceptional slope $\alpha \in \F E$ is the smallest natural number $q$ such that $\alpha = \varepsilon(\frac{p}{2^q})$. Using the order, we rewrite $$ C' = C \cap (-1,1) = (-1,1) \sm \bigcup_{\alpha\in \F E} I_\alpha = \bigcap_{q=0}^\infty \left((-1,1) \sm \bigcup_{\alpha\in \F E \atop \ord \alpha = q} I_\alpha\right).$$ Define sets $$C_n = \bigcap_{q=0}^n \left((-1,1) \sm \bigcup_{\alpha\in \F E \atop \ord \alpha = q} I_\alpha\right),$$ putting $C_{-1} = (-1,1)$ for convenience.  Then $C_0$ is a union of two intervals $$C_0 = [-1+x_{-1},0-x_0]\cup[0+x_0, 1- x_1].$$ If $x\in C$, then  $x$ is either in the left or right interval of $C_0$.  To form $C_1$, one removes an interior open subinterval from each component of $C_0$, yielding a left and right subinterval of each component of $C_0$.  At each stage, $C_n$ is comprised of $2^{n+1}$ disjoint intervals, with each component of $C_{n-1}$ containing two of the intervals of $C_n$, a ``Left'' and a ``Right'' subinterval.  Associated to an $x\in C'$, we build an infinite word $\sigma_x$ in the alphabet $\Sigma = \{L,R\}$; the $n$th letter of $\sigma_x$ corresponds to whether $x$ lies in the Left or Right part of the component $I\subset C_{n-1}$ containing $x$.

Conversely, given an infinite word $\sigma\in \Sigma^\N$, there is a unique $x_\sigma \in C'$ with $\sigma_{x_\sigma}= \sigma$.  Thus points in $C'$ determine and are determined by left-right sequences.

\subsection{Action of the alphabet $\Sigma = \{L,R\}$ on $\F E'$} Let us develop a more concrete description of the points in the set $C'$.

An exceptional slope $\gamma\in \F E$ can be uniquely written as $\gamma = \varepsilon(\frac{p}{2^q})$, where $q\geq 0$ and $p$ is odd unless $q=0$.  Then setting $$ \alpha = \varepsilon\left(\frac{p-1}{2^q}\right) \qquad \textrm{and} \qquad \beta = \varepsilon\left(\frac{p+1}{2^q}\right)$$ we have $\gamma = \alpha.\beta$.  We call $\alpha$ and $\beta$ the \emph{left} and \emph{right parent slopes} of $\gamma$, and write $$\alpha = \pare_L(\gamma) \qquad \textrm{and} \qquad \beta = \pare_R(\gamma).$$
We define $$\gamma\cdot  L =  \pare_L(\gamma).\gamma = \alpha.\gamma \qquad \textrm{and} \qquad \gamma \cdot R = \gamma.\pare_R(\gamma)=\gamma.\beta.$$ 

Denoting by $\Sigma^*$ the monoid of finite strings in the alphabet $\Sigma$, we induce a right action of $\Sigma^*$ on $\F E$.  The next easy lemma establishes that exceptional slopes can either be understood in terms of dyadic rationals via the map $\varepsilon$ or in terms of finite left-right sequences.

\begin{lemma}
There is a bijection $\Sigma^*\to \F E'$ given by $$\sigma \mapsto 0\cdot \sigma.$$ That is, every $\gamma\in \F E'$ admits a unique finite left-right sequence $\sigma_\gamma \in \Sigma^\ast$ such that $$\gamma = 0 \cdot \sigma_\gamma.$$
\end{lemma}

In light of this correspondence, we extend the notion of parents to strings in $\Sigma^\ast$.  Given $\sigma\in \Sigma^\ast$, we write $\pare_L(\sigma) $ for the sequence corresponding to the left parent of $0\cdot \sigma$; we similarly define $\pare_R(\sigma)$.

The action of $\Sigma^*$ on $\F E'$ can be naturally extended to a function $\F E'\times \Sigma^\N \to C'$.  Specifically, for $\gamma\in \F E'$ and $\sigma\in \Sigma^\N$ we define $\gamma \cdot \sigma$ to be the point in the Cantor set associated to the concatenated left-right sequence $\sigma_\gamma \sigma$; in symbols, $$\gamma\cdot \sigma = x_{\sigma_\gamma \sigma}.$$  As a special case, we note the identity $$0 \cdot \sigma=x_\sigma.$$  This action has a natural continuity property: if $\sigma|_n$ denotes the length $n$ initial segment of $\sigma$, then $$ \gamma \cdot \sigma = \lim_{n\to\infty} \gamma \cdot \sigma|_n$$ in the real numbers.  The key point is that the exceptional slope $0\cdot \sigma|_n$ lies in the same component of $C_{n-1}$ as $x_\sigma$ does for each $n$.

\begin{example}
Consider the right endpoint $\alpha+x_\alpha$ of an interval $I_\alpha$, where $\alpha\in \F E'$.  Using a bar to denote an infinite repeating block, we have $$\alpha + x_\alpha = \alpha\cdot R\overline{L}.$$ Similarly, $$\alpha - x_\alpha = \alpha \cdot L\overline{R}.$$ It follows that the endpoints of intervals $I_\alpha$ correspond precisely to left-right sequences  $\sigma \in \Sigma^{\N}$ which are \emph{eventually constant}.
\end{example}

To prove the theorem, we will show that if the left-right sequence $\sigma_x$ of an $x\in C''$ is not eventually constant, then its continued fraction expansion is not eventually repeating.

\begin{remark}
If $x\in C''=C\cap (0,\frac{1}{2})$, then the word $\sigma_x\in \Sigma^\N$ begins with $RL$. Conversely, if $\sigma\in \Sigma^\N$ is a word beginning with $RL$, then $0\cdot \sigma$ lies in $C''$.  Similarly, if $\gamma\in \F E'' = \F E\cap (0,\frac{1}{2})$, then the word $\sigma_\gamma\in \Sigma^\ast$ begins with $RL$, and any $\sigma\in \Sigma^\ast$ beginning with $RL$ has $0\cdot \sigma \in\F E''.$
\end{remark}

\begin{example}\label{parentExample}
Let $\gamma\in \F E''$.  It is easy to describe $\pare_L(\sigma_\gamma)$ and $\pare_R(\sigma_\gamma)$ in terms of $\sigma_\gamma$.  Since $\gamma\in \F E''$, the sequence $\sigma_\gamma$ begins with $RL$, and it is  not a constant sequence.  We can  write either $$\sigma_\gamma = \sigma' RL^n \qquad \textrm{or} \qquad \sigma_\gamma = \sigma' L R^n$$ for some word $\sigma'\in \Sigma^\ast$ and some $n\geq 1$.  In the first case, where $\sigma_\gamma$ ends in $L$, we have $$\pare_L(\sigma_\gamma)= \sigma' \qquad \textrm{and} \qquad \pare_R(\sigma_\gamma)=\sigma'RL^{n-1}.$$ Analogously, if $\sigma_\gamma$ ends in $R$, then $$\pare_L(\sigma_\gamma) = \sigma' L R^{n-1} \qquad \textrm{and} \qquad \pare_R(\sigma_\gamma) = \sigma'.$$
\end{example}

One qualitative consequence of the example is important enough to single out.

\begin{corollary}\label{truncationCor}
If $\gamma\in \F E''$, then the parent slopes of $\gamma$ have associated left-right sequences which are initial segments of $\sigma_\gamma$.
\end{corollary}

\subsection{Continued fractions}\label{subsec-continued} The continued fraction expansions of exceptional slopes have a remarkable  inductive structure, which we recall here.

Given real numbers $a_0,a_1,\ldots,a_k$, we write $[a_0;a_1,\ldots,a_k]$ for the continued fraction $$[a_0;a_1,\ldots,a_k] := a_0 + \cfrac{1}{a_1+
  \cfrac{1}{\ddots \raisebox{-1.2ex}{${}+\cfrac{1}{a_k}$}}}
 $$ assuming this expression is defined.  Any rational number $\alpha\in \Q$ has two finite continued fraction expansions as above where $a_0\in \Z$ and all the $a_i$ with $i>0$ are positive integers.  One of them $\alpha = [a_0;a_1,\ldots,a_k]$ has $a_k=1$, and the other is given by $\alpha = [a_0;a_1,\ldots,a_{k-2},a_{k-1}+1].$

For convenience, we work with the set $\F E''$ of exceptional slopes in the interval $(0,\frac{1}{2})$.  If $\alpha \in \F E''$ and $\alpha = [a_0;a_1,\ldots,a_k]$, then clearly $a_0=0$.  It is useful to consider the string of integers $a_1a_2\cdots a_k \in (\N_{>0})^*$.   We write $\alpha^e \in (\N_{>0})^*$ for the \emph{even} length string $a_1\cdots a_k$ of natural numbers such that $\alpha = [0;a_1,\ldots,a_k]$.  Similarly, we write $\alpha^o\in (\N_{>0})^*$ for the \emph{odd} length string with this property.

\begin{theorem}[{\cite[Theorem 3.2]{HuizengaPaper2}}]\label{cfracThm}
Let $\gamma\in \F E'' = \F E \cap (0,\frac{1}{2})$ be an exceptional slope, and put $\alpha = \pare_L(\gamma)$ and $\beta = \pare_R(\gamma)$.  Then $$\gamma^e=\beta^o 2\alpha^e.$$ That is, the even length continued fraction expansion of $\gamma$ is obtained from the odd length expansion of $\beta$ by concatenating a $2$ and the even length expansion of $\alpha$.

Furthermore,
\begin{enumerate}
\item every number in $\gamma^e$ and $\gamma^o$ is a $1$ or a $2$, so that these words actually are in $\{1,2\}^*$, and
\item the word $\gamma^e$ is a palindrome.
\end{enumerate}
\end{theorem}

By Corollary \ref{truncationCor}, the left-right sequences of the parents of $\gamma$ are initial segments of the left-right sequence of $\gamma$.  This fact makes it very easy to compute the continued fraction of an exceptional slope from its left-right sequence.  As a basis for the inductive algorithm, one uses that $0^e$ is the empty string and $(\frac{1}{2})^o$ is the string $2$.

\begin{example}\label{cfracEx}
Consider the word $RLLLRR$ and the corresponding exceptional slope $\gamma = 0\cdot RLLLRR$.  We  compute the continued fraction expansion of $\gamma$ by computing the continued fraction expansions of the exceptional slopes corresponding to initial segments.  

\begin{center}
\begin{longtable}{ccccc}
\toprule 
$\sigma$ & $\pare_L(\sigma)$ & $\pare_R(\sigma)$ & $(0\cdot \sigma)^e$ & $(0\cdot \sigma)^o$ \\\midrule
\endfirsthead
\multicolumn{5}{l}{{\small \it continued from previous page}}\\
\toprule
$\sigma$ & $\pare_L(\sigma)$ & $\pare_R(\sigma)$ & $(0\cdot \sigma)^e$ & $(0\cdot \sigma)^o$ \\\midrule \endhead
\bottomrule \multicolumn{5}{r}{{\small \it continued on next page}} \\ \endfoot
\bottomrule
\endlastfoot

$\emptyset$ & N/A& N/A & $\emptyset$ & N/A\\
$R$ & $\emptyset$ & N/A & $11$ & $2$\\
$RL$ & $\emptyset$ & $R$ & $22$ & $211$\\
$RLL$ & $\emptyset$ & $RL$ & $2112$ & $21111$\\
$RLLL$ & $\emptyset$ & $RLL$ & $211112$ & $2111111$\\
$RLLLR$ & $RLLL$ & $RLL$ & $211112211112$ & $2111122111111$\\
$RLLLRR$ & $RLLLR$ & $RLL$ & $211112211112211112$ & $2111122111122111111$\\
\end{longtable}
\end{center}
We summarize our various points of view of $\gamma$ by concluding \begin{align*}\gamma &= 0 \cdot RLLLRR \\&= \varepsilon\left(\frac{1}{2}-\frac{1}{4}-\frac{1}{8}-\frac{1}{16}+\frac{1}{32}+\frac{1}{64}\right) \\&= [0;2,1,1,1,1,2,2,1,1,1,1,2,2,1,1,1,1,2]\\&= \frac{19760}{51641}.  \end{align*}
\end{example}

\begin{example}
It is also easy to systematically compute the continued fraction expansions of all exceptional slopes with order at most $k$ for any fixed $k$.  Figure \ref{figure-slope} shows the computation of the slopes in $[0, \frac{1}{2}]$ up to order 5. In the diagram, the order increases from left to right. Except for those of order 5, we have recorded both the even and odd length continued fraction expansions of the exceptional slopes. We always write the odd length expansion above the even length expansion. To avoid cluttering the diagram,  we have only recorded the even length expansions for slopes of order 5. To find the continued fraction of an exceptional slope $\alpha$, trace its left-right sequence along the solid lines, starting from the expansion $\emptyset$ of $0$.

\begin{figure}[htbp]
\begin{center}
\input{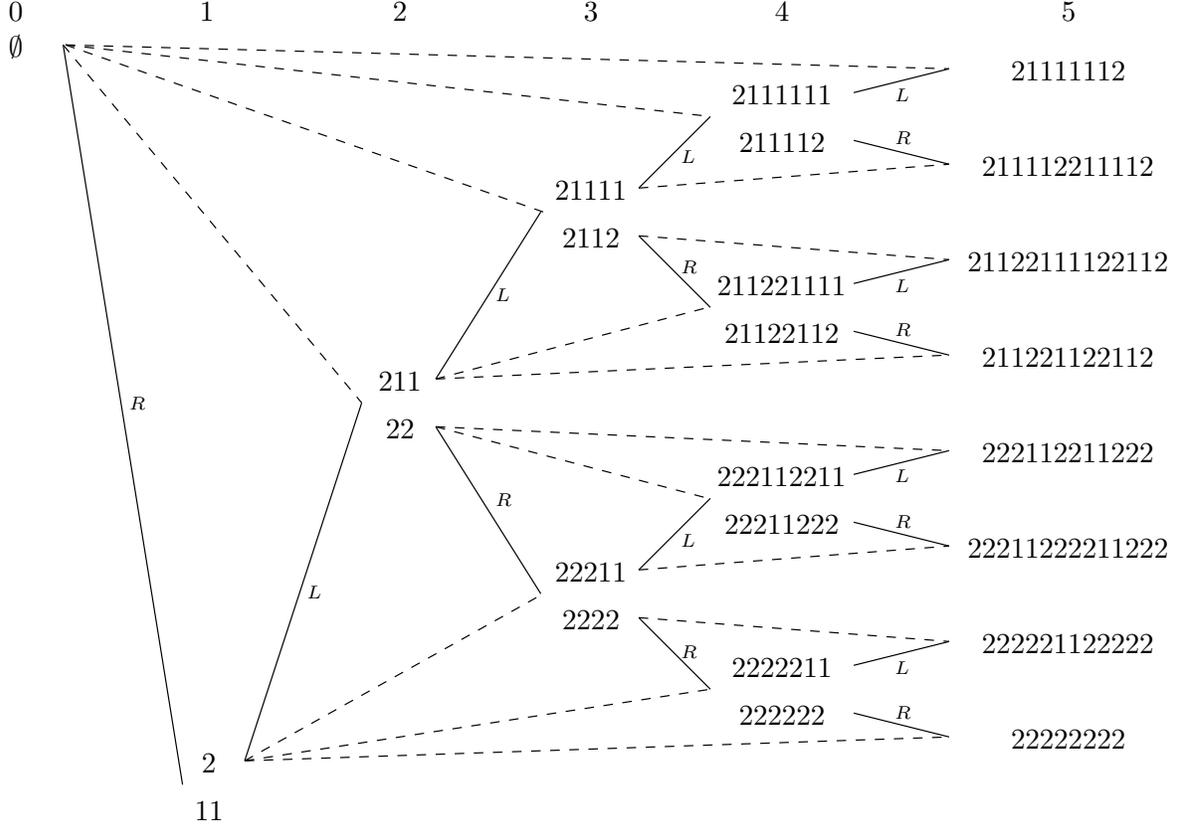}
\end{center}
\caption{The continued fraction expansions of exceptional slopes of order at most $5$.  Orders are listed at the top of the diagram.}
\label{figure-slope}
\end{figure}
\end{example}

\subsection{Repeating structures} We say a word $\sigma = a_1a_2\ldots a_{k}$ in some alphabet $S$ has \emph{period $p>0$} if $a_i = a_{i+p}$ for all $1\leq i\leq k-p$.  In this case, $\sigma$ is the length $k$ initial segment of the infinite word $\overline{a_1\ldots a_p}$.

\begin{lemma}\label{repeatingLemma}
Let $\sigma\in S^*$ have length $k$, and suppose $p$ is the smallest period of $\sigma$.  If $\sigma$ has another period $p'$ and $p+p' \leq k$, then $p$ divides $p'$.
\end{lemma}
\begin{proof}
Suppose $p$ does not divide $p'$, and write $p' = pq+r$ with $0<r<p$.  We claim $\sigma$ has period $r$, violating the minimality of $p$.  Write $\sigma = a_1\ldots a_k$, and let $i$ be an index with $1\leq i \leq k-r$.  There is some $m\geq 0$ such that $1 \leq i-mp \leq p$.  Then using the two periods of $\sigma$, $$a_i = a_{i-mp} = a_{i-mp+p'}=a_{i-mp+p'-(q-m)p}=a_{i+r},$$ where the hypothesis $p+p'\leq k$ ensures $i-mp+p' \leq k.$ 
\end{proof}

Many exceptional slopes have continued fraction expansions exhibiting a simple repeating structure.

\begin{proposition}\label{repeatingStructProp}
Let $\gamma\in \F E''$ be an exceptional slope such that $\sigma_\gamma$ ends in an $R$.  Write $$\sigma_{\gamma} = \sigma' LR^n.$$ If $\beta = \pare_R(\gamma)$ and $\alpha = \pare_L(\beta)$ (caution: $\alpha$ is the parent of $\beta$, not of $\gamma$), then $$\gamma^e = (\beta^o2)^{n+1} \alpha^e.$$ The word $\alpha^e$ is an initial segment of $\beta^o$.  Thus the length of $\beta^o2$ is a period of $\gamma^e$.
 
 Furthermore, unless $\beta = \frac{1}{2}$, the length of $\beta^o2$ is the smallest period of $\gamma^e$.  If $\beta = \frac{1}{2}$, then $\gamma^e = 2^{2n+2}$.
\end{proposition}
\begin{proof}
The proof is an analysis of the continued fraction algorithm.  We have $\sigma_\beta = \sigma'$ as in Example \ref{parentExample}.  Observe that $\pare_L(\sigma_\beta L) = \sigma_\alpha$ and $\pare_R(\sigma_\beta L) = \sigma_\beta$, while for $n\geq 1$ we have $$\pare_L(\sigma_\beta LR^{n}) = \sigma_\beta L R^{n-1} \qquad \textrm{and} \qquad \pare_R(\sigma_\beta LR^n) = \sigma_\beta.$$ Using this data one computes the continued fraction expansion of $\gamma$ using Theorem \ref{cfracThm} as in the following table.
\begin{center}
\begin{tabular}{ccccc}
\toprule 
$\sigma$ & $\pare_L(\sigma)$ & $\pare_R(\sigma)$ & $(0\cdot \sigma)^e$ & $(0\cdot \sigma)^o$ \\\midrule
$\sigma_\beta$ & $\sigma_\alpha$  &  & $\beta^e$ & $\beta^o$\\
$\sigma_\beta L$ & $\sigma_\alpha$ & $\sigma_\beta$ & $\beta^o2 \alpha^e$ & \\
$\sigma_\beta L R$ & $\sigma_\beta L$ & $\sigma_\beta$ & $(\beta^o2)^2\alpha^e$\\
$\sigma_\beta L RR$ & $\sigma_\beta LR$ & $\sigma_\beta$ & $(\beta^o2)^3\alpha^e$\\
$\vdots$ & $\vdots$ & $\vdots $ & $\vdots$\\
$\sigma_\beta L R^n$ & $\sigma_\beta LR^{n-1}$ & $\sigma_\beta$ & $(\beta^o2)^{n+1}\alpha^e$\\ \bottomrule
\end{tabular}
\end{center}

Let us show that the tail $\alpha^e$ of the continued fraction expansion of $\gamma$ is an initial segment of $\beta^o$.  If $\beta = \frac{1}{2}$, then $\alpha = 0$ and there is nothing to prove, so assume $0<\alpha < \beta < \frac{1}{2}$. In this case, we recall $$\beta = \alpha.\pare_R(\beta).$$ Therefore, $$\beta^e=(\pare_R(\beta))^o2\alpha^e.$$ The words $\alpha^e$ and $\beta^e$ are both palindromes (Theorem \ref{cfracThm}), so $\alpha^e$ is a \emph{proper} initial segment of $\beta^e$.  Since $0< \beta < \frac{1}{2}$ the first digit in $\beta^e$ is a $2$, and thus its last digit is also a $2$.  Then $\beta^o$ is obtained from $\beta^e$ by replacing its final digit with $11$, and we conclude $\alpha^e$ is also an initial segment of $\beta^o$.

Suppose $p$ is the length of the smallest period of $\gamma^e$.  By Lemma \ref{repeatingLemma}, $p$ divides the length of $\beta^o2$.  If $p$ is smaller than this length, we can write $\beta^o2$ as a concatenation of $k\geq 2$ copies of some word $\delta$: $$\beta^o2 = \delta^k.$$ Under the assumption $0<\beta < \frac{1}{2}$, the expansion $\beta^e$ is a palindrome beginning and ending in $2$; write $$\beta^e = 2 b_2b_3\ldots b_{l-1} 2.$$ Then $$\beta^o 2 = 2 b_2 b_3 \ldots b_{l-1}112.$$ The word $\delta$ must then begin with $2$ and end in $112$, so be of the form $$\delta = 2 d_2\ldots d_{m-3}112.$$ Defining $$\delta' = 2 d_2\ldots d_{m-3}2,$$ we have $$\beta^e = \delta^{k-1} \delta'.$$ Expanding, $$\beta^e = (2 d_2\ldots d_{m-3}\underline{1}12) \delta^{k-2} (\underline{2} d_2\ldots d_{m-3}2).$$ Recalling that $\beta^e$ is a palindrome, we obtain a contradiction by comparing the underlined digits in the previous expression.
\end{proof}

\subsection{Continued fractions of points in $C''$}  Irrational numbers have unique continued fraction expansions.  For an $x\in C''$, the expansion takes the form $[0;a_1,a_2,\ldots]$, with all $a_i\in \{1,2\}$ since $x$ is a limit of exceptional slopes.  We write $x^f$ for the word $a_1a_2\ldots\in \{1,2\}^\N$.

\begin{lemma}
Let $\gamma\in \F E''$, and let $\sigma\in \Sigma^*$ be nonempty.
\begin{enumerate}
\item If the first letter of $\sigma$ is $R$, then $\gamma^e$ is a proper initial segment of $(\gamma\cdot \sigma)^e$.
\item If the first letter of $\sigma$ is $L$, then $\gamma^o$ is a proper initial segment of $(\gamma\cdot \sigma)^e$.
\end{enumerate}
\end{lemma}
\begin{proof}
Induct on the length $k$ of $\sigma$.  If $\sigma = R$, then $$(\gamma\cdot R)^e = (\gamma . \pare_R(\gamma))^e=(\pare_R(\gamma))^o2\gamma^e.$$ Since $(\gamma\cdot R)^e$ is a palindrome, $\gamma^e$ is a proper initial segment of $(\gamma\cdot R)^e$.  Similarly, if $\sigma = L,$ then $\gamma^o$ is a proper initial segment of $(\gamma\cdot L)^e$.

For the inductive step, without loss of generality assume $\sigma$ begins with $R$ and has length $k+1$.  Let $\sigma' = \sigma|_k$ be the length $k$ initial segment of $\sigma$.  By induction, $\gamma^e$ is a proper initial segment of $(\gamma\cdot \sigma')^e$.  Furthermore, either $(\gamma\cdot \sigma')^e$ or $(\gamma\cdot \sigma')^o$ is a proper initial segment of $(\gamma \cdot \sigma)^e$ by the previous paragraph.  As $\gamma^e$ is a proper initial segment of $(\gamma\cdot \sigma')^e$, it is also a proper initial segment of $(\gamma\cdot \sigma')^o$.  Thus in either case, $\gamma^e$ is a proper initial segment of $(\gamma\cdot \sigma)^e$.
\end{proof}

We then obtain the following result by continuity.  The continued fraction expansion of $x\in C''$ can be understood in terms of the fractions of exceptional slopes approximating it.

\begin{corollary}\label{CantorFracCor} Let $x\in C''$, and let $\sigma_x$ be the left-right sequence corresponding to $x$.  Let $\sigma'$ be an initial segment of $\sigma_x$, and let $\gamma = 0\cdot \sigma'$.
\begin{enumerate}
\item If $\sigma'$ precedes an $R$ in $\sigma$, then $\gamma^e$ is an initial segment of $x^f$.

\item If $\sigma'$ precedes an $L$ in $\sigma$, then $\gamma^o$ is an initial segment of $x^f$.
\end{enumerate}
\end{corollary}

Since even length expansions of exceptional slopes are palindromes, we deduce the next corollary.

\begin{corollary}\label{pallindromeCor}
If $x\in C''$ is not an endpoint of an interval $I_\alpha$, then $x^f$ begins in arbitrarily long palindromes.
\end{corollary}

Having studied the continued fractions of exceptional slopes in great detail, the next result is now easy.

\begin{theorem}\label{thm-cantor-notquad} If $x\in C''$ has a left-right sequence $\sigma_x$ which is not eventually constant, then $x^f$ is not eventually repeating.  Thus $x$ is not a quadratic irrational.
\end{theorem}
\begin{proof}
Assume $x^f$ is eventually repeating with period $p$, beginning from the $c$th digit of $x^f$.  Fix $N$ very large relative to $p$ and $c$.  Since $\sigma_x$ is not eventually constant, we may find an initial segment $\sigma'$ of $\sigma_x$ which ends in $R$ such that $\sigma'$ and the parents of $\sigma'$ have length at least $N$.  Correspondingly, let $\alpha$ and $\beta$ be the parents of $\gamma = 0\cdot \sigma'$, so that $\gamma = \alpha.\beta$.    As in Proposition \ref{repeatingStructProp}, we have $$\gamma^e = (\beta^o2)^{n}\alpha^e$$ for some $n\geq 2$.  Then $\beta^o$ and $\alpha^e$ have length at least $N$.  Either $\gamma^e$ or $\gamma^o$ is an initial segment of $x^f$ by Corollary \ref{CantorFracCor}, so $x^f$ begins like $$x^f = \beta^o2\beta^o2\ldots.$$
 Let $p'>N$ be the length of $\beta^o2$.  Writing $x^f = a_1a_2\ldots$, consider the word $$\sigma = a_ca_{c+1}\ldots a_{2p'}.$$  By construction, $\sigma$ has periods $p$ and $p'$, and $\sigma$ has length longer than $p+p'$.  Lemma \ref{repeatingLemma} shows that $p$ is in fact the shortest period of $\sigma$. Again by Lemma \ref{repeatingLemma}, $p$ divides $p'$ and the word $\beta^o2$ is a power of a shorter word.  As in the proof of Proposition \ref{repeatingStructProp}, this is a contradiction.
\end{proof}

Theorem \ref{cantorTranscendThm} follows at once from Corollary \ref{pallindromeCor}, Theorem \ref{thm-cantor-notquad} and Theorem \ref{thm-AdB}.

\section{The Beilinson spectral sequence}\label{sec-beilinson}

In this section, we use the generalized Beilinson spectral sequence to describe a canonical resolution of the general sheaf in a moduli space $M(\xi)$.  This resolution illuminates the structure of the effective cone.

\subsection{Generalities on the Beilinson spectral sequence}  
Following \cite{Drezet}, a \emph{triad} is a triple $(E,G,F)$ of exceptional bundles where the slopes are of the form $(\alpha,\alpha.\beta,\beta)$, $(\beta-3,\alpha,\alpha.\beta)$, or $(\alpha.\beta,\beta,\beta+3)$ for some exceptional slopes $\alpha,\beta$ of the form \begin{equation}\tag{$*$}\label{eqn-ab} \alpha = \varepsilon\left(\frac{p\vphantom{1}}{2^q}\right) \qquad \beta = \varepsilon\left(\frac{p+1}{2^q}\right).\end{equation} Corresponding to the triad $(E,G,F)$ is a fourth exceptional bundle $M$ defined as the cokernel of the canonical map $\ev^*:G\to F\te \Hom(G,F)^*$.  The collection $(E^*(-3),M^*,F^*)$ is another triad.

The bundles of a triad form a \emph{strong exceptional collection} for the bounded derived category $D^b(\P^2)$.  In other words, if $A,B$ are two bundles in a triad with $A$ listed before $B$, then $\Ext^i(A,B)=0$ for $i>0$ and $\Ext^i(B,A)=0$ for all $i$.  Furthermore, a triad is a full collection, in the sense that it generates the derived category.  This fact is a consequence of the generalized Beilinson spectral sequence, which we now recall.

\begin{theorem}[\cite{DrezetBeilinson}]\label{beilinsonThm}
Let $U$ be a coherent sheaf, and let $(E,G,F)$ be a triad.  Write
$$\begin{aligned}
G_{-2} &= E \\
G_{-1} &= G \\
G_0 &= F 
\end{aligned}\qquad\qquad
\begin{aligned} 
F_{-2} &= E^*(-3)\\
F_{-1} &= M^*\\
F_{0} &= F^*,
\end{aligned}$$ and put $G_i=F_i=0$ if $i\notin \{-2,-1,0\}$.
There is a spectral sequence with $E_1^{p,q}$-page $$E_{1}^{p,q} = G_p\te H^q(U\te F_p)$$ which converges to $U$ in degree $0$ and to $0$ in all other degrees.
\end{theorem}

By choosing the triad appropriately, it is possible to obtain interesting information about the sheaf $U$.  To study $\Eff(M(\xi))$, we will choose a triad related to the corresponding exceptional bundle.

\subsection{Choosing the triad} 
The corresponding exceptional slope to $\xi$ (see Definition \ref{def-exceptional-slope}) can be written uniquely in the form $\alpha.\beta$, where $\alpha,\beta$ are as in Equation ($*$).  Denote by $U$ a general sheaf in $M(\xi)$.  As discussed in \S\ref{sec-exceptional}, the qualitative behavior of the effective  cone depends on the sign of $(\xi,\xi_{\alpha.\beta})$. Which triad we use to decompose $U$ depends on the sign of this pairing as well.

\begin{enumerate} \item In case $(\xi,\xi_{\alpha.\beta})>0$, we will decompose $U$ according to the triad $(E_{-\alpha-3},E_{-\beta},E_{-(\alpha.\beta)})$.

\item If $(\xi,\xi_{\alpha.\beta})<0$, we will use the triad $(E_{-(\alpha.\beta)-3},E_{-\alpha-3},E_{-\beta})$.

\item When $(\xi,\xi_{\alpha.\beta})=0$, the spectral sequence from either of the previous two triads degenerates and gives the same result.
\end{enumerate}

We will primarily focus on the first case in this section, briefly summarizing the results in the other two cases at the end.

\subsection{The spectral sequence when $(\xi,\xi_{\alpha.\beta})>0$} Assuming $(\xi,\xi_{\alpha.\beta})>0$, we aim to describe the Beilinson spectral sequence of the general $U\in M(\xi)$ with respect to the triad $(E_{-\alpha-3},E_{-\beta},E_{-(\alpha.\beta)})$.  In the notation of Theorem \ref{beilinsonThm}, we have $$\begin{aligned}
G_{-2} &= E_{-\alpha-3} \\
G_{-1} &= E_{-\beta}  \\
G_0 &= E_{-(\alpha.\beta)} 
\end{aligned}\qquad\qquad
\begin{aligned} 
F_{-2} &= E_\alpha\\
F_{-1} &= E_{\alpha.(\alpha.\beta)}\\
F_{0} &= E_{\alpha.\beta}.
\end{aligned}$$
(To compute $F_{-1}$, notice that the only exceptional bundle $F_{-1}$ such that $(E_{\alpha},F_{-1},E_{\alpha.\beta})$ is a triad is $F_{-1}=E_{\alpha.(\alpha.\beta)}$.) Then the associated spectral sequence has the following first page.

$$
\xymatrix{E_{-\alpha-3}\te H^2(U\te E_\alpha) \ar[r] & E_{-\beta}\te H^2(U\te  E_{\alpha.(\alpha.\beta)})  \ar[r] & E_{-(\alpha.\beta)}\te H^2(U\te E_{\alpha.\beta}) \\
E_{-\alpha-3}\te H^1(U\te  E_\alpha) \ar[r] & E_{-\beta}\te H^1(U\te E_{\alpha.(\alpha.\beta)})  \ar[r] & E_{-(\alpha.\beta)}\te H^1(U\te E_{\alpha.\beta})  \\
E_{-\alpha-3}\te H^0(U\te  E_\alpha) \ar[r] & E_{-\beta}\te H^0(U\te E_{\alpha.(\alpha.\beta)})  \ar[r] & E_{-(\alpha.\beta)}\te H^0(U\te E_{\alpha.\beta})  \\
}$$

Of course, we expect that many of these terms vanish.  To see which vanishings are expected, we perform a numerical calculation.

\begin{lemma}\label{lem-estimates}
With the notation in this subsection, the following inequalities hold.
\begin{enumerate}
\item $\mu(U\te E_\alpha)> -3/2$
\item $\chi(U\te E_\alpha)<0$
\item $\chi(U\te E_{\alpha.(\alpha.\beta)})<0$
\end{enumerate}
\end{lemma}
\begin{proof}
(1) The reference parabola $Q_{\xi_0}$ corresponding to $\OO_{\P^2}$ passes through the points $(0,1)$ and $(-3,1)$ in the $(\mu,\Delta)$-plane.  It follows that any parabola obtained by translating $Q_{\xi_0}$ downward by more than $\frac{1}{2}$ intersects the line $\Delta = \frac{1}{2}$ in two points more than $3$ units apart.

Let $\mu_0\in\R $ be the larger of the two slopes such that $(\mu_0,\frac{1}{2})\in Q_\xi$, and let $(\mu_1,\Delta_1)\in \Q^2$ be the vertex of $Q_\xi$.  The parabola $Q_\xi$ is obtained from $Q_{\xi_0}$ by translation by $(-\mu(\xi),-\Delta(\xi))$.  Since $\Delta(\xi)>\frac{1}{2}$, the previous paragraph shows that 
$$\mu_0 - \mu_1 > \frac{3}{2}.$$  The vertex of $Q_{\xi_0}$ occurs at $(-\frac{3}{2},-\frac{1}{8})$, so $\mu_1 = -\mu(\xi)-\frac{3}{2}$. By construction of the corresponding exceptional slope, $\mu_0 \in I_{\alpha.\beta}$, so $\mu_0 < \alpha.\beta+x_{\alpha.\beta}$.  Then $$\mu(U\te E_\alpha) = \mu(\xi)+\alpha=-\frac{3}{2}-\mu_1+\alpha> \alpha- \mu_0 > \alpha - (\alpha.\beta+ x_{\alpha.\beta}) \geq \frac{\sqrt 5-5}{2} \approx -1.38 > -\frac{3}{2},$$ with the worst case of the ``$\geq$'' occurring  when $\alpha$ and $\alpha.\beta$ are consecutive integers.

(2 \& 3) Consider the parabola $Q_{\xi_{-(\alpha.\beta)}}$ giving the arc of the curve $\delta(\mu)=\Delta$ over the interval $[\alpha.\beta-x_{\alpha.\beta},\alpha.\beta]$.  The parabola $Q_\xi$ intersects $Q_{\xi_{-(\alpha.\beta)}}$ at some point $(\mu_2,\Delta_2)$, and $\mu_2 > \alpha.\beta-x_{\alpha.\beta}$ by Lemma \ref{stupidInequalityLemma}.  We have $\mu_0<\mu_2$, so to the left of $\mu_2$ the parabola $Q_{\xi_{-(\alpha.\beta)}}$ lies above the parabola $Q_\xi$.  Therefore, if $\zeta$ is any Chern character with invariants on $Q_{\xi_{-(\alpha.\beta)}}$ such that $\mu(\zeta)< \mu_2$, we observe $(\xi,\zeta)<0$.

The invariants of $E_\alpha$ and $E_{\alpha.(\alpha.\beta)}$ both lie on $Q_{\xi_{-(\alpha.\beta)}}$ since $\chi(E_{\alpha.\beta},E_\alpha)=\chi(E_{\alpha.(\alpha.\beta)},E_\alpha)=0.$  The required inequalities then follow from the previous paragraph.
\end{proof}

As a consequence of part (1) Lemma \ref{lem-estimates} and semistability, the $H^2$-terms in the spectral sequence vanish for \emph{any} $U\in M(\xi)$.  We expect that for the general $U\in M(\xi)$ the $H^0$- and $H^1$-terms are governed by the Euler characteristic.  That is, we expect the bundles $U\te E_\alpha$, $U\te E_{\alpha.(\alpha.\beta)}$, and $U\te E_{\alpha.\beta}$ are all nonspecial, so that the spectral sequence takes the form
$$\xymatrix{
E_{-\alpha-3}\te \C^{m_1} \ar[r] & E_{-\beta}\te \C^{m_2}  \ar[r] & 0  \\
0\ar[r] & 0  \ar[r] & E_{-(\alpha.\beta)}\te \C^{m_3}  \\
}$$
where
$$\begin{aligned}
m_1 &=-\chi(U\te E_\alpha)\\
m_2 &=-\chi(U\te E_{\alpha.(\alpha.\beta)})\\
m_3 &=\chi(U\te E_{\alpha.\beta}).
\end{aligned}
$$

\begin{proposition}\label{resolutionProp}
With the notation of this subsection, if $U\in M(\xi)$ is general, then the Beilinson spectral sequence degenerates as above.  

More precisely, the general $U\in M(\xi)$ admits a resolution of the form $$0\to E_{-\alpha-3}^{m_1}\xrightarrow{\psi} E_{-\beta}^{m_2}\oplus E_{-(\alpha.\beta)}^{m_3} \to U\to 0.$$ 
\end{proposition}
\begin{proof}
Consider a bundle $U$ defined by an exact sequence $$0\to E_{-\alpha-3}^{m_1}\xrightarrow{\psi} E_{-\beta}^{m_2} \oplus E_{-(\alpha.\beta)}^{m_3}\to U\to 0,$$ where the map $$\psi\in \Hom(E_{-\alpha-3}^{m_1},E_{-\beta}^{m_2}\oplus E_{-(\alpha.\beta)}^{m_3})$$ is general.

\emph{Step 1:} show $\ch(U)=\xi$.  At this point we do not know $\psi$ is injective.  We, therefore, virtually compute the Chern character of $U$ under this assumption.  That is, we compute the Chern character of the mapping cone of $\psi$ in the derived category.
 
To do this, we use the fact that the Beilinson spectral sequence for a sheaf $U\in M(\xi)$ converges to $U$.  Specifically, suppose $A,B,C$ are sheaves and we have a spectral sequence with $E_1$-page
$$\xymatrix{
A \te \C^{m_{11}} \ar[r] & B \te \C^{m_{21}} \ar[r] & C \te \C^{m_{31}}\\
A \te \C^{m_{12}} \ar[r] & B \te \C^{m_{22}} \ar[r] & C \te \C^{m_{32}}\\
}
$$that converges to a sheaf $U$ in degree $0$ and $0$ in other degrees.  An elementary calculation shows $$\ch(U) = -(m_{11}-m_{12})\ch(A)+(m_{21}-m_{22})\ch(B)-(m_{31}-m_{32})\ch(C).$$ In the present situation, we conclude $$\xi = -m_1\ch(E_{-\alpha-3})+m_2\ch(E_{-\beta})+m_3\ch(E_{-\alpha.\beta}).$$

\emph{Step 2:} show $\psi$ is injective.  The sheaves $$\sHom(E_{-\alpha-3},E_{-\beta}) \qquad \textrm{and} \qquad \sHom(E_{-\alpha-3},E_{-(\alpha.\beta)})$$ are both globally generated, from which we find $$\sHom(E_{-\alpha-3}^{m_1},E_{-\beta}^{m_2}\oplus E_{-(\alpha.\beta)}^{m_3})$$ is globally generated as well.  Combining a Bertini-type theorem \cite[Proposition 2.6]{HuizengaPaper2} with our calculation of the virtual rank of $E$, we conclude $\psi$ is injective.

\emph{Step 3:} show $U$ has the expected zeroes in its Beilinson spectral sequence.  It is enough to check $$H^0(U\te E_{\alpha})=H^0(U\te E_{\alpha.(\alpha.\beta)})=H^1(U\te E_{\alpha.\beta})=0,$$ and these vanishings follow immediately from orthogonality properties of exceptional bundles.  For instance, $\sHom(E_\beta,E_\alpha)$ and $\sHom(E_{\alpha.\beta},E_\alpha)$ both have no cohomology, while $\Ext^1(E_{\alpha+3},E_\alpha)\cong \Ext^1(E_\alpha,E_\alpha)^*=0$ since exceptional bundles are rigid.  This gives $H^0(U\te E_\alpha)=0$; the other vanishings are proved similarly.

\emph{Step 4:} show $U$ is stable.  For brevity, write $A = E_{-\alpha-3}^{m_1}$ and $B = E_{-\beta}^{m_2} \oplus E_{-(\alpha.\beta)}^{m_3}.$  Let $$S\subset \Hom(A,B)$$ be the open subset parameterizing injective sheaf maps with torsion-free cokernels.  

The set $S$ is in fact nonempty; to see this we show the general cokernel $U$ is torsion-free.    The argument in Step 2 shows the map $\psi_x$ on fibers is an injection outside of codimension $2$ in case $r( \xi) = 1$ and an injection everywhere if $r( \xi) >1$. Hence, any torsion in $U$ occurs in codimension at least $2$.  If $U$ has torsion in codimension $2$, let $T\subset U$ be its torsion subsheaf.  Then $T(-k)$ has sections for all large $k$, so $U(-k)$ has sections for all large $k$.  Since any locally free sheaf $E$ has $H^0(E(-k))=H^1(E(-k))=0$ for large $k$, this contradicts the defining sequence for $U$, so $U$ is torsion-free.

Consider the family $\F U/S$ of quotients parameterized by $S$. We claim that $\F U$ is a complete family of \emph{prioritary} sheaves.  A prioritary sheaf is a torsion-free sheaf $U$ such that $$\Ext^2(U,U(-1))=0.$$ Dr\'{e}zet \cite{DrezetBeilinson} shows that if $\alpha<\beta$ are two exceptional slopes then $\Ext^i(E_\alpha,E_\beta)=0$ for $i>0$.  From this vanishing theorem, it follows easily that the sheaves parameterized by $S$ are all prioritary.

To see $\F U/S$ is a complete family, fix an $s\in S$, let $U_s$ be the corresponding prioritary sheaf, and consider the Kodaira-Spencer map $$\Hom(A,B) = T_s S \to \Ext^1(U_s,U_s);$$ we must show this map is surjective.  The Kodaira-Spencer map factors as the composition $$\Hom(A,B)\xrightarrow{i} \Hom(A,U_s)\xrightarrow{j} \Ext^1(U_s,U_s),$$ where the maps come from the natural long exact sequences  $$\Hom(A,B)\xrightarrow{i} \Hom(A,U_s) \to \Ext^1(A,A)$$ and $$\Hom(A,U_s)\xrightarrow{j} \Ext^1(U_s,U_s) \to \Ext^1(B,U_s).$$  From the exact sequence $$\Ext^1(B,B) \to \Ext^1(B,U_s) \to \Ext^2(B,A)$$ and the orthogonality properties of exceptional bundles, we conclude $\Ext^1(A,A) = \Ext^1(B,U_s) = 0$, so $i$, $j$, and the Kodaira-Spencer map are all surjective.

The Artin stack $\cP(\xi)$ of prioritary sheaves with Chern character $\xi$ is an irreducible stack containing the stack $\cM(\xi)$ of semistable sheaves with Chern character $\xi$ as a dense open subset \cite{HirschowitzLaszlo}.  It follows that the general sheaf in $M(\xi)$ is parameterized by a point of $S$.
\end{proof}

From the perspective of Bridgeland stability, it is better to interpret the resolution of a general $U\in M(\xi)$ discussed in the previous proposition from the viewpoint of the derived category.

\begin{theorem}\label{triangleThm}
With the notation of this subsection, let $U\in M(\xi)$ be general.  Let $W\in D^b(\P^2)$ be the mapping cone of the canonical evaluation map $$E_{-(\alpha.\beta)}\te \Hom(E_{-(\alpha.\beta)},U) \to U,$$ so that there is a distinguished triangle $$E_{-(\alpha.\beta)} \te \Hom(E_{-(\alpha.\beta)},U) \to U \to W \to \cdot.$$ Then $W$ is isomorphic to a complex of the form $$E_{-\alpha-3} \te \C^{m_1} \to E_{-\beta} \te \C^{m_2}$$ sitting in degrees $-1$ and $0$.  Any two complexes of this form which are isomorphic to $W$ are in the same orbit under the natural action of $\GL(m_1)\times \GL(m_2)$  on the space of such complexes.
\end{theorem}
\begin{proof}
It is easy to show that if $$0\to A\to B\oplus C \to D\to 0$$ is an exact sequence of sheaves, then the mapping cone of $C\to D$ is isomorphic to the complex $A\to B$ sitting in degrees $-1$ and $0$.  The uniqueness statement is implied by the next lemma.
\end{proof}

\begin{lemma}\label{complexMap}
Consider a pair of two-term complexes
$$W = E_{-\alpha-3} \te \C^{m_1}\to E_{-\beta} \te \C^{m_2}$$ and
$$W' = E_{-\alpha-3} \te \C^{m_1'}\to E_{-\beta} \te \C^{m_2'},$$ each sitting in degrees $-1$ and $0$.  Every homomorphism $W\to W'$ in the derived category $D^b(\P^2)$ is realized by a homomorphism of complexes, so $$\Hom_{D^b(\P^2)}(W,W') = \Hom_{\Kom(\P^2)}(W,W').$$
\end{lemma}
\begin{proof}
Throughout the proof, we write $\Hom(-,-)$ for $\Hom_{D^b(\P^2)}(-,-)$. Let $A=E_{-\alpha-3}\te \C^{m_1}$ and $B = E_{-\beta} \te \C^{m_{2}}$, and similarly define $A',B'$ for $W'$.  We take the maps $A\to B$ and $A'\to B'$ as implicit.  

We have a distinguished triangle $$B\to W\to A[1]\to\cdot,$$ and similarly for $W'$.  We obtain a long exact sequence $$0\to \Hom(A[1],W')\to \Hom(W,W')\to \Hom(B,W')\to \Hom(A,W').$$ (The left zero comes from $\Hom(B[1],W')=0$, which follows from the exact sequence $$\Ext^{-1}(B,B)=\Hom(B[1],B)\to \Hom(B[1],W')\to \Hom(B[1],A[1]) = \Hom(B,A)$$ by properties of exceptional bundles.)  Next, the sequence $$0= \Hom(A[1],B')\to \Hom(A[1],W')\to \Hom(A[1],A'[1])\to \Hom(A[1],B'[1])$$ identifies $\Hom(A[1],W')$ with the space of homomorphisms $A\to A'$ such that the composition $$A\to A'\to B'$$ is zero; that is, homomorphisms $A[1]\to W'$ are just the homomorphisms of underlying complexes.

On the other hand, homomorphisms $A\to W'$ in the derived category are \emph{not} the same thing as homomorphisms of the underlying complexes.  We have a sequence $$ \Hom(A,A') \to \Hom(A,B')\to \Hom(A,W') \to \Hom(A,A'[1])=0.$$ Here $\Hom(A,B')$ is the space of homomorphisms of complexes $A\to W'$; such a homomorphism $A\to B'$ induces the zero homomorphism in $\Hom(A,W')$ if it factors as $A\to A'\to B'$.

Finally, we identify the space $\Hom(B,W')$ as the space of homomorphisms of underlying complexes, as the exact sequence $$0=\Hom(B,A)\to \Hom(B,B')\to \Hom(B,W')\to \Hom(B,A[1])=0$$ gives an isomorphism $\Hom(B,W')\cong \Hom(B,B')$.  Then the image of the map $\Hom(W,W')\to \Hom(B,W')$ is the kernel of $\Hom(B,W')\to \Hom(A,W')$.  This kernel is identified with those homomorphisms $B\to B'$ such that the composition $A\to B\to B'$ is zero when viewed as giving a homomorphism $A\to W'$; in light of the preceding paragraph, this means precisely that the composition $A\to B'$ factors as $A\to A'\to B'$.  Thus the image of $\Hom(W,W')\to \Hom(B,W')$ is identified with those maps of complexes 
$$\xymatrix{
& A'\ar[d]\\
B \ar[r] & B'}$$
which can be lifted to maps of complexes $$\xymatrix{
A \ar[r]\ar[d]& A'\ar[d]\\
B \ar[r] & B'}$$ As the kernel of $\Hom(W,W')\to \Hom(B,W')$ consists precisely of maps of complexes $$\xymatrix{
A \ar[r]\ar[d]& A'\ar[d]\\
B \ar[r]^0 & B'},$$ we conclude $\Hom(W,W')$ consists precisely of the maps of complexes.
\end{proof}

\begin{remark} For experts on Bridgeland stability, we remark that the resolution we have found arises from Bridgeland stability.  We refer the reader to \cite[\S 8 and \S 9]{HuizengaPaper2} for the necessary definitions and  results.  We use the notation in that paper. The precise theorem is as follows.

\begin{theorem}\label{bridgelandThm}
Let $U\in M(\xi)$ be general.  Let  $\Lambda = W(U,E_{-(\alpha.\beta)})$ be the wall in the $st$-plane of stability conditions where $U$ and $E_{-(\alpha.\beta)}$ have the same $\mu_{s,t}$-slope.  There is an exact sequence $$0\to E_{-(\alpha.\beta)}\te \Hom(E_{-(\alpha.\beta)},U) \to U \to W\to 0$$ in the corresponding categories $\cA_s$.  The kernel and cokernel objects are $\sigma_{s,t}$-semistable outside of walls nested inside $\Lambda$, so $U$ is $\sigma_{s,t}$-semistable along and outside $\Lambda$, but not stable inside $\Lambda$.  The wall $\Lambda$ has center $(s_0,0)$ satisfying $s_0 = -\mu^+-\frac{3}{2}$, where $\mu^+$ is the corresponding orthogonal slope to $\xi$.
\end{theorem}
\begin{proof}
As in \S 9 of \cite{HuizengaPaper2}, it is easy to show that $E_{-(\alpha.\beta)}$ and $W$ are $\sigma_{s,t}$-semistable for $(s,t)$ outside semicircular walls of radius at most $\sqrt{5}/2$.  It is also easy to see that each of these walls is either nested inside or outside $\Lambda$.

The only step of the proof of Theorem 8.5 from \cite{HuizengaPaper2} which does not generalize immediately is the proof that $\Lambda$ has radius at least $\sqrt{5}/2$. Furthermore, recent results simplify the proof of this fact greatly.  We will see later that the corresponding orthogonal Chern character $\xi^+$ satisfies $$(\xi^+,\xi)=(\xi^+,\xi_{-(\alpha.\beta)})=0.$$ Then by Proposition 4.1 from \cite{CoskunHuizenga}, the wall $\Lambda$ is the  semicircular wall with center $(s,0)$ and radius $\rho >0$ given by $$\mu^+ = -s-\frac{3}{2} \qquad \textrm{and} \qquad 2\Delta^+ = \rho^2-\frac{1}{4}.$$  The inequality $\rho > \sqrt{5}/2$ follows at once from $\Delta^+>\frac{1}{2}$.
\end{proof}
\end{remark}

\subsection{The spectral sequence when $(\xi,\xi_{\alpha.\beta})<0$.}  As mentioned earlier, in this case we use the spectral sequence for $U$ associated to the triad $(E_{-(\alpha.\beta)-3},E_{-\alpha-3},E_{-\beta})$.  This time, we have 
 $$\begin{aligned}
G_{-2} &= E_{-(\alpha.\beta)-3} \\
G_{-1} &= E_{-\alpha-3}  \\
G_0 &= E_{-\beta} 
\end{aligned}\qquad\qquad
\begin{aligned} 
F_{-2} &= E_{\alpha.\beta}\\
F_{-1} &= E_{(\alpha.\beta).\beta}\\
F_{0} &= E_{\beta}.
\end{aligned}$$
It is easily shown that the general $U$ has a spectral sequence of the form 
$$\xymatrix{
E_{-(\alpha.\beta)-3}\te \C^{m_3} \ar[r] & 0 \ar[r] & 0  \\
0\ar[r] & E_{-\alpha-3}\te \C^{m_1}  \ar[r] & E_{-\beta}\te \C^{m_2}  \\
}$$
with
$$\begin{aligned}
m_1 &=\chi(U\te E_{(\alpha.\beta).\beta})\\
m_2 &=\chi(U\te E_{\beta}).\\
m_3 &=-\chi(U\te E_{\alpha.\beta})
\end{aligned}
$$
The analog of Theorem \ref{triangleThm} is  substantially nicer in this case, since the analog of the object $W$ of the derived category in Theorem \ref{triangleThm} is  a sheaf.

\begin{theorem}
With the notation of this subsection, let $U\in M(\xi)$ be general, and let $$E_{-\alpha-3}\te \C^{m_1}\to E_{-\beta}\te \C^{m_2}$$ be the complex appearing in the generalized Beilinson spectral sequence for $U$, sitting in degrees $-1$ and $0$.  The map in this complex is injective, so the complex is isomorphic to a sheaf $W$ with resolution $$0\to E_{-\alpha-3}\te \C^{m_1} \to E_{-\beta}\te \C^{m_2}\to W\to 0.$$  Any two resolutions of $W$ of this form are conjugate under the $\GL(m_1)\times \GL(m_2)$ action.  Finally, the sheaf $U$ fits in a triangle $$W\to U \to E_{-(\alpha.\beta)-3}[1] \te \Hom(U,E_{-(\alpha.\beta)-3}[1])^*\to \cdot,$$ where the second map is the canonical one.  
\end{theorem}

The obvious analog of Theorem \ref{bridgelandThm} holds here as well.

\subsection{The spectral sequence when $(\xi,\xi_{\alpha.\beta})=0$}\label{ssec-equality}  In this case, either of the previous spectral sequences degenerate to an exact sequence $$0\to E_{-\alpha-3}\te \C^{m_1}\to E_{-\beta}\te \C^{m_2} \to U\to 0,$$ where $m_1 = -\chi(U\te E_\alpha)$ and $m_2 = \chi(U\te E_\beta)$.  Comparing the two different spectral sequences shows we  have equalities 
$$\begin{aligned}-\chi(U \te E_{\alpha}) &= \chi(U\te E_{(\alpha.\beta).\beta})\\ -\chi(U\te E_{\alpha.(\alpha.\beta)}) &= \chi(U\te E_\beta).\end{aligned}$$
  We have chosen to represent each $m_i$ by its simplest expression.  Furthermore, the next result follows easily from analyzing the first spectral sequence for $U$.

\begin{proposition}\label{exceptionalDivProp}
With the notation of this subsection, a sheaf $U\in M^s(\xi)$ lies in the Brill-Noether locus $D_{E_{\alpha.\beta}}$ if and only if it does not have a resolution of the form $$0\to E_{-\alpha-3} \te \C^{m_1}\to E_{-\beta}\te \C^{m_2}\to U\to 0.$$ Furthermore, $D_{E_{\alpha.\beta}}$ is a divisor.
\end{proposition}
\begin{proof}
By orthogonality properties of exceptional bundles, any stable sheaf $U$ with a resolution of the above form has $H^0(U\te E_{\alpha.\beta})=0$, and thus does not lie in $D_{E_{\alpha.\beta}}$.  This also shows $D_{E_{\alpha.\beta}}$ is a divisor.  Conversely, if $H^0(U\te E_{\alpha.\beta})=0$, then the first spectral sequence for $U$ takes the form $$\xymatrix{
E_{-\alpha-3}\te \C^{m_1+d_1} \ar[r] & E_{-\beta}\te \C^{m_2+d_2}  \ar[r] & 0  \\
E_{-\alpha-3}\te \C^{d_1} \ar[r] & E_{-\beta}\te \C^{d_2}\ar[r] & 0  \\
}$$
for some nonnegative integers $d_1,d_2$.  If either $d_1$ or $d_2$ is positive, the sequence cannot converge to a sheaf in degree $0$, so we conclude $U$ has a resolution of the expected form.
\end{proof}

\section{The Kronecker fibration}\label{sec-kronecker}

The generalized Beilinson spectral sequence described in the previous section allows us to describe a natural rational map from a moduli space $M(\xi)$ to a moduli space of Kronecker modules.  These moduli spaces have Picard rank $1$, and analyzing this fibration will allow us to determine the effective cone of $M(\xi)$.

\subsection{Moduli spaces of Kronecker modules}\label{ssec-Kronecker}  Let $V$ be an $N$-dimensional vector space with $N\geq 3$.  A linear map $e:B\te V\to A$ for $b$- and $a$-dimensional vector spaces $B$, $A$, respectively, can be viewed as a representation of the Kronecker quiver $Q$ consisting of $2$ vertices and $N$ arrows from the first vertex to the second.
$$\xymatrix{
\bullet \ar@<-3.5ex>[r] \ar@<-2.5ex>[r] \ar@<3.5ex>[r]^N \ar@<2.5ex>[r]_\vdots & \bullet 
}$$
We call such a representation a \emph{Kronecker $V$-module}.  
The \emph{dimension vector} $\udim e\in \N^2$ of  such a Kronecker module $e$ is the vector $(b,a)$.  If $f:B'\te V\to A'$ is another Kronecker $V$-module, a homomorphism $f\to e$ is a pair of maps $\beta: B'\to B$ and $\alpha: A'\to A$ making the square
$$\xymatrix{
B \te V \ar[r] & A\\
B' \te V \ar[u]^{\beta \te \id}\ar[r] & A' \ar[u]^\alpha
}$$
commute.  The space of Kronecker $V$-modules forms an abelian category.

Fix a dimension vector $(b,a)$ and vector spaces $B,A$ of dimensions $b,a$.  Every Kronecker $V$-module $e$ with $\udim e = (b,a)$ is isomorphic to an element of $\Hom_\C(B\te V,A)$.  The natural action of $\C^*$ on this space preserves the isomorphism type of the corresponding Kronecker modules, and there is an induced action of $\SL(B)\times \SL(A)$ on $\P \Hom_\C(B\te V,A)$.  A Kronecker $V$-module $e$ is \emph{(GIT)-semistable} if any (equivalently, all) corresponding points of $\P \Hom_\C(B\te V,A)$ are semistable points of this action.  We denote by $Kr_N(\udim e)$ the quotient space of semistable modules. 

The \emph{expected dimension} of $Kr_N(b,a)$ is $$\edim Kr_N(b,a)= abN-a^2-b^2+1.$$ It is a theorem of Dr\'ezet \cite{Drezet} that if the expected dimension of $Kr_N(b,a)$ is positive then stable modules with vector $(b,a)$ exist, and thus $Kr_N(b,a)$ has the expected dimension.  

For our present work, it is important to note that semistability of Kronecker modules can be detected by the existence of orthogonal modules.  To explain this, we define the  \emph{Euler characteristic} of a pair $(f,e)$ of modules by $$\chi(f,e) = \hom_Q(f,e) - \ext^1_Q(f,e);$$ all the higher Ext terms vanish.  The Euler characteristic of a pair can be computed in terms of the numerical invariants of the module and the quiver.  Precisely, if $\udim f = (b',a')$ and $\udim e = (b,a)$, then $$\chi(f,e) = b'b+a'a-Nb'a.$$ It is useful to note that $$\edim Kr_N(\udim e)=1-\chi(e,e).$$

\begin{theorem}\label{orthogThm}\cite{Schofield,SvdB}
A Kronecker $V$-module $e$ is semistable if and only if it admits a semistable \emph{left orthogonal module}, i.e. there is some semistable Kronecker $V$-module $f$ such that $$\Hom_Q(f,e) = \Ext^1_Q(f,e) = 0.$$ Semistability is also equivalent to the existence of a semistable right orthogonal module.
\end{theorem}

Suppose $f$ and $e$ are Kronecker modules with  $\chi(f,e)=0$.  If $\udim e = (b,a)$ then $$\udim f = k(a,Na-b)$$ for some $k\in \Q_{> 0}$, so that the dimension vector of $f$ is determined up to a scalar multiple by $\udim e$.  We also observe that $\edim Kr_N(\udim f)$ is positive if and only if $\edim Kr_N(\udim e)$ is positive, since
$$\edim Kr_N(\udim f)-1 = k^2 (\edim Kr_N(\udim e)-1).$$

It is shown in Dr\'ezet \cite{Drezet} that the Kronecker spaces $Kr_N(\udim e)$ have Picard group $\Z$.  If $f$ is a left orthogonal module to $e$, then the locus $$D_f = \{e':\Hom_Q(f,e')\neq 0\} \subset Kr_N(\udim e)$$ is an effective generator of $\Pic Kr_N(\udim e)\te \Q$ such that $e$ is not in the base locus of $D_f$.  

\subsection{Kronecker modules coming from complexes}\label{ssec-inducedModule}  Suppose $(E,G,F)$ is a triad of exceptional bundles.  Consider the category of complexes $V$ of the form $$E^{n_1}\to G^{n_2}, \qquad (n_i\in \N)$$ where the homomorphisms are maps of complexes. Corresponding to such a complex, we can form a Kronecker $\Hom(E,G)^*$-module $$\C^{n_1}\te \Hom(E,G)^*\to \C^{n_2}.$$ Since exceptional bundles are simple, homomorphisms of complexes are the same  as homomorphisms of the corresponding Kronecker modules.  In fact, this correspondence is an equivalence of categories, and it preserves the Euler characteristic of pairs.

\subsection{Moduli spaces of sheaves birational to Kronecker spaces}\label{ssec-birational} 
In this subsection we study some simple moduli spaces of sheaves on the plane.  The orthogonal invariants $(\mu^+,\Delta^+)$  of an arbitrary moduli space $M(\xi)$ are always of this form unless they are the invariants of an exceptional bundle, so studying these spaces is particularly important.  These moduli spaces also include the height zero moduli spaces of Picard rank $1$.

\begin{proposition}\label{birationalProp}
Fix an exceptional slope $\alpha.\beta$, and suppose $\xi$ is a stable Chern character with $(\xi,\xi_{-(\alpha.\beta)})=0$ and $\Delta(\xi)> \frac{1}{2}$.  We either have $\mu(\xi) > \alpha.\beta-x_{\alpha.\beta}$ or $\mu(\xi)< \alpha.\beta-3+x_{\alpha.\beta}$.  Let $N =\chi(E_{\beta-3},E_\alpha)$.
 
\begin{enumerate}
\item If $\mu(\xi)>\alpha.\beta-x_{\alpha.\beta}$, the general $V\in M(\xi)$ admits a resolution of the form $$0\to E_{\beta-3}^{n_1}\to E_{\alpha}^{n_2} \to V\to 0$$ unique up to the action of $\GL(n_1)\times \GL(n_2)$, and $M(\xi)$ is birational to $Kr_N(n_1,n_2)$.  

\item If $\mu(\xi)<\alpha.\beta-3+x_{\alpha.\beta}$ and $r(\xi)\geq 2$, the general $V\in M(\xi)$ admits a resolution $$0 \to V \to E_{\beta-3}^{n_1}\to E_{\alpha}^{n_2} \to 0$$ unique up to the action of $\GL(n_1)\times \GL(n_2)$, and $M(\xi)$ is birational to $Kr_N(n_1,n_2).$
\end{enumerate} 
\end{proposition}
\begin{proof}
The invariants of $\xi$ lie on the parabola $Q_{\xi_{-(\alpha.\beta)}}$, which intersects the line $\Delta=\frac{1}{2}$ at the points $\alpha.\beta-3+x_{\alpha.\beta}$ and $\alpha.\beta-x_{\alpha.\beta}$.  Since $\Delta(\xi) >\frac{1}{2}$, it follows that either $\mu(\xi) > \alpha.\beta-x_{\alpha.\beta}$ or $\mu(\xi)< \alpha.\beta-3+x_{\alpha.\beta}$.

(1) For $\mu(\xi) > \alpha.\beta - x_{\alpha.\beta}$, it is easy to show that the associated exceptional slope to $\xi$ is $-(\alpha.\beta)=(-\beta).(-\alpha)$.   The results of \S \ref{ssec-equality} then give that there is a resolution of the general $V\in M(\xi)$ of the form $$0\to E_{\beta-3}^{n_1}\to E_{\alpha}^{n_2}\to V\to 0,$$ where $n_1 = -\chi(E_\beta,V)$ and $n_2=\chi(E_\alpha,V)$.  We also know that the map in a general complex $$E_{\beta-3}^{n_1}\to E_{\alpha}^{n_2}$$ is injective with semistable cokernel of Chern character $\xi$ (Proposition \ref{resolutionProp}), and that if two general semistable bundles with Chern character $\xi$ are isomorphic then the isomorphism is induced by an isomorphism of their resolutions (Lemma \ref{complexMap}).  

We now associate to a two-term complex as above a Kronecker $\Hom(E_{\beta-3},E_{\alpha})^*$-module $e$ with dimension vector $(n_1,n_2)$ as in Section \ref{ssec-inducedModule}.  Then $$\dim(M(\xi)) = 1-\chi(V,V) = 1-\chi(e,e)=\edim Kr_N(n_1,n_2),$$ so since $\dim(M(\xi))> 0$ we conclude $Kr_N(n_1,n_2)$ has positive expected dimension.  Thus the general such module is stable.  Furthermore, isomorphism of Kronecker modules corresponds to isomorphism of complexes, so we obtain a birational map $M(\xi)\dashrightarrow Kr_N(n_1,n_2)$.

(2) Since the rank is at least 2, the general $V\in M(\xi)$ is locally free and taking duals gives a birational map $M(\xi) \dashrightarrow M(\xi^*).$  This case then follows by applying case (1) to the Chern character $\xi^\ast $ satisfying $(\xi^\ast,\xi_{\alpha.\beta-3})=0.$
\end{proof}

\subsection{The Kronecker fibration}  Let $\xi$ be a stable Chern character corresponding to a moduli space of semistable sheaves with Picard rank $2$.  Let $\alpha.\beta$ be the corresponding exceptional slope, and let  $\xi^+$ be the corresponding orthogonal character. 

Let $U$ be a general sheaf in $M(\xi)$.  We saw in \S\ref{sec-beilinson} that $U$  determines a complex $W$  of the form $$E_{-\alpha-3}^{m_1}\to E_{-\beta}^{m_2},$$ and that the isomorphism class of this complex depends only on $U$.  Conversely, for a general such complex $W$, there exist stable bundles $U$ with associated complex $W$ by Proposition \ref{resolutionProp}.  The complex $W$ then determines a $\Hom(E_{-\alpha-3},E_{-\beta})^*$-module with dimension vector $(m_1,m_2)$.  Putting $N = \hom (E_{-\alpha-3},E_{-\beta})$, we therefore should have a dominant rational map $$M(\xi) \dashrightarrow Kr_N(m_1,m_2).$$ The only thing we haven't verified yet is that the Kronecker module associated to a general $W$ is semistable.

\begin{proposition}\label{fibrationProp}
With notation as above, the Kronecker moduli space $Kr_N(m_1,m_2)$ is positive-dimensional. Thus there is a dominant rational map $$\pi:M(\xi)\dashrightarrow Kr_N(m_1,m_2).$$
\begin{enumerate}
\item If $(\xi,\xi_{\alpha.\beta})\neq 0$, then the general fiber of $\pi$ is positive-dimensional.
\item If $(\xi,\xi_{\alpha.\beta})=0$, then $\pi$ is birational.
\end{enumerate}
\end{proposition}
\begin{proof}
By Proposition \ref{stableCharacterProp}, the Chern character $\xi^+$ is semistable. Let $V$ be a general semistable vector bundle with character $\xi^+$. To prove $Kr_N(m_1,m_2)$ is positive-dimensional, first suppose $(\xi,\xi_{\alpha.\beta})>0$.  By definition, $\{(\mu^+,\Delta^+)\}= Q_{\xi_{-(\alpha.\beta)}}\cap Q_\xi$, and $\mu^+>\alpha.\beta-x_{\alpha.\beta}$ by Lemma \ref{stupidInequalityLemma}.  According to Proposition \ref{birationalProp}, the bundle $V$ has a resolution  $$0\to E_{\beta-3}^{n_1}\to E_{\alpha}^{n_2}\to V\to 0.$$ By contrast, if $U\in M(\xi)$ is general, we have a distinguished triangle $$E_{-(\alpha.\beta)}^{m_3}\to U\to W\to \cdot,$$ where $W$ is a complex $$W: E_{-\alpha-3}^{m_1}\to E_{-\beta}^{m_2}.$$ As $(\mu^+,\Delta^+)\in Q_\xi$, we have $\chi(V^*,U)=0$.  Orthogonality properties of exceptional bundles show $\chi(V^*,E_{-(\alpha.\beta)})=0$ as well, and therefore $\chi(V^*,W)=0$.  By Serre duality, this is equivalent to $\chi(W,V^*(-3)) = 0$, and we further obtain $\chi(W,V^*(-3)[1])=0$. As $V^*(-3)[1]$ is isomorphic in the derived category to a complex $$V^*(-3)[1]:E_{-\alpha-3}^{n_1}\to E_{-\beta}^{n_2}$$ sitting in degrees $-1$ and $0$, both $W$ and $V^*(-3)[1]$ correspond to Kronecker $\Hom(E_{-\alpha-3},E_{-\beta})^*$-modules, say $f$ and $e$, respectively.
We conclude that $\chi(f,e)=0$, so $\udim f$ is a left-orthogonal dimension vector to $\udim e$.  The space $M(\ch V)$ is positive dimensional, so Proposition \ref{birationalProp} shows that $Kr_N(n_1,n_2)$ is positive-dimensional.  By the discussion at the end of \S\ref{ssec-Kronecker}, $Kr_N(m_1,m_2)$ is positive-dimensional as well.

In case $(\xi,\xi_{\alpha.\beta})<0$, a similar argument works.  When $(\xi,\xi_{\alpha.\beta})=0$, the positive-dimensionality and part (2) were proved in Proposition \ref{birationalProp}.

Let us show that if $(\xi,\xi_{\alpha.\beta})\neq 0$ then the general fiber of $\pi$ is positive-dimensional.  We focus on the case $(\xi,\xi_{\alpha.\beta})>0$, with the other case being similar.  Since $\pi$ is dominant, we know $\dim M(\xi) \geq \dim Kr_N(n_1,n_2)$, and we must show the inequality is strict.  Computing dimensions, this inequality is equivalent to$$\chi(U,U)\leq \chi(W,W).$$ Expanding each side by additivity of the Euler characteristic, we have the equivalent inequality $$m_3( \xi,\xi_{\alpha.\beta})+\chi(W,U)\leq \chi(W,U)-m_3\chi(W,E_{-(\alpha.\beta)} ).$$ Since $m_3=(\xi,\xi_{\alpha.\beta})>0$, this is the same as $$(\xi,\xi_{\alpha.\beta})\leq - \chi(W,E_{-(\alpha.\beta)})=m_1 \chi(E_{-\alpha-3},E_{-(\alpha.\beta)})-m_2\chi(E_{-\beta},E_{-(\alpha.\beta)}).$$ Recalling \begin{align*} m_1 &= -(\xi,\xi_\alpha)\\ m_2 &= - (\xi,\xi_{\alpha.(\alpha.\beta)}),\end{align*} we see that there are nonzero constants $C_1,C_2$ (independent of $\xi$) such that the inequality is equivalent to $$(\xi,\xi_{\alpha.\beta}) \leq (\xi,C_1 \xi_{\alpha}+C_2 \xi_{\alpha.(\alpha.\beta)}).$$ Noting that $\{\xi_\alpha,\xi_{\alpha.(\alpha.\beta)},\xi_{\alpha.\beta}\}$ form a basis for $K(\P^2)\te \R$, we finally conclude there is some nonzero $\zeta \in K(\P^2)\te \R$ such that the inequality $\dim M(\xi)\geq \dim Kr_N(n_1,n_2)$ is equivalent to $$(\xi,\zeta)\geq 0.$$ Here $\zeta$ depends only on the corresponding exceptional slope $\alpha.\beta$, and not on the particular Chern character $\xi$.  We must verify the inequality is strict. 

Consider the set $$\Omega = \left\{\xi': \begin{array}{c}\Delta(\xi') > \delta(\mu(\xi')),\\ \textrm{$\alpha.\beta$ is the corresponding exceptional slope to $\xi'$,} \\ \textrm{and $(\xi',\xi_{\alpha.\beta})>0$}\end{array}\right\} \subset (K(\P^2)\te \Q )\sm \{0\}.$$ Clearly $\Omega$ is an open set.  If it intersects the plane $\zeta^\perp$, then $\Omega$ contains some rational Chern character $\xi'$ such that $(\xi',\zeta)<0$, a contradiction.  Thus $\Omega$ and $\zeta^\perp$ must be disjoint, and we conclude $(\xi,\zeta)>0$.  This implies $\dim M(\xi) > \dim Kr_N(n_1,n_2)$.
\end{proof}

\section{The effective cone of the moduli space}\label{sec-effective}

In this section we prove our main theorem on the effective cone of divisors on $M(\xi)$ by making use of the Kronecker fibration.

\subsection{The primary edge of the effective cone}\label{ss-primaryEffective}  Let $\xi$ be a stable Chern character of positive rank such that $M(\xi)$ has Picard rank $2$.  Let $\alpha.\beta$ be the associated exceptional slope.  We first consider the non-exceptional case.

\begin{theorem}\label{effThm}
Suppose $(\xi,\xi_{\alpha.\beta})\neq 0$.  Let $\xi^+$ be a primary corresponding orthogonal character to $\xi$.  If $V\in M(\xi^+)$ is general, then the Brill-Noether locus $$D_V = \{U\in M(\xi):h^0(U\te V)\neq 0\}$$ is a divisor spanning the primary edge of the effective  cone of $M(\xi)$.  Under the isomorphism $\NS(M(\xi))\cong \xi^\perp$, it corresponds to $\xi^+$.
\end{theorem}
\begin{proof}
We consider the case $(\xi,\xi_{\alpha.\beta})>0$, with the other case being similar.  A general $U\in M(\xi)$ fits into a triangle $$E_{-(\alpha.\beta)}^{m_3}\to U\to W \to \cdot,$$ while $V$ has a resolution $$0\to E_{\beta-3}^{n_1}\to E_{\alpha}^{n_2} \to V\to 0.$$ We wish to show $V\te U$ has no cohomology.  We find $\sHom(E_{\alpha.\beta},V)$ has no cohomology.  As in the first part of the proof of Proposition \ref{fibrationProp}, we reduce to showing $\sHom(W,V^*(-3)[1])$ has no cohomology.  But if $f$, $e$ are the Kronecker modules corresponding to $W$ and $V^*(-3)[1],$ respectively, this is equivalent to the vanishing of $\Hom(f,e)$, which is true by Theorem \ref{orthogThm}.  Then $D_V$ is just the pullback of the ample divisor $D_f \subset Kr(\udim e)$ under the map $\pi: M(\xi)\dashrightarrow Kr(\udim e)$.  It follows that $D_V$ spans an edge of the effective cone.  The computation of the class of $D_V$ and the fact that it spans the primary edge of the effective cone follows from Proposition \ref{BNProp} and Corollary \ref{primaryCor}.
\end{proof}

\begin{corollary}\label{movingCor}
If $(\xi,\xi_{\alpha.\beta})\neq 0$, then the primary edges of the effective and movable cones of divisors of $M(\xi)$ coincide.
\end{corollary}
\begin{proof}
Since $D_V$ is the pullback of an ample divisor from the space of Kronecker modules, its stable base locus is contained in the indeterminacy locus of $\pi$.  As $M(\xi)$ is normal, this locus has codimension at least $2$.  Thus $D_V$ spans a movable class.
\end{proof}

The exceptional case requires a slightly different argument.

\begin{theorem}\label{excEffThm}
Suppose $(\xi,\xi_{\alpha.\beta})=0$.
The Brill-Noether divisor $D_{E_{\alpha.\beta}}$ spans the primary edge of the effective cone of $M(\xi)$.  Furthermore, $D_{E_{\alpha.\beta}}$ is reduced and irreducible.  Under the isomorphism $\NS(M(\xi))\cong \xi^\perp$, the divisor corresponds to the exceptional Chern character $\xi_{\alpha.\beta}$.
\end{theorem}
\begin{proof}
By Proposition \ref{exceptionalDivProp}, $D_{E_{\alpha.\beta}}$ is an effective divisor, and it is the complement of the locus of sheaves $U$ with a resolution of the expected form $$0\to E_{-\alpha-3}^{m_1}\to E_{-\beta}^{m_2} \to U\to 0.$$ To see that it is extremal, we produce a dual moving curve.  Let $$S = \P\Hom(E_{-\alpha-3}^{m_1},E_{-\beta}^{m_2}),$$ and let $\F U/S$ be the universal cokernel sheaf.  It is a complete family of prioritary sheaves with $\Delta>\delta(\mu)$, so by \cite[Lemma 18.3.1]{LePotierLectures} the open set $S' \subset S$ consisting of those $s$ such that $U_s$ is stable has complement of codimension at least $2$.  Thus for a general sheaf $U\in M(\xi)$ we can find a complete curve in $S'$ passing through a point $s$ corresponding to $U$.  Any such curve is disjoint from $D_{E_{\alpha.\beta}}$, so this divisor is extremal.

The divisor $D_{E_{\alpha.\beta}}$ is reduced and irreducible because it is extremal and $\xi_{\alpha.\beta}$ is a generator of $K(\P^2) \cap \Q\xi_{\alpha.\beta}$.  Indeed, $c_1(\xi_{\alpha.\beta})$ and $r_{\alpha.\beta}$ are coprime, so no integral class $\zeta$ of rank smaller than $r_{\alpha.\beta}$ has slope $\alpha.\beta$.
\end{proof}

In the exceptional case, the primary edge of the movable cone of divisors is considerably more challenging to describe.  We will report on progress on this question in future work.

\subsection{The rank $0$ case}\label{ssec-rank0}  In this subsection, we discuss the primary edge of the effective cone of a moduli space $M(\xi)$ of semistable rank $0$ sheaves.  We normalize our Chern character $\xi = (0,d,\chi)$ in terms of the first Chern class $d$ and Euler characteristic $\chi$.  Assuming the Picard rank is $2$, we have $d\geq 3$.  The general point $U\in M(\xi)$ corresponds to a line bundle of Euler characteristic $\chi$ supported on a smooth plane curve of degree $d$.

The primary difference with the positive rank case is that the parabolas $Q_\xi$ from the positive rank case degenerate to vertical lines in the $(\mu,\Delta)$-plane.  As before, we define a locus $$Q_\xi = \{(\mu,\Delta): (1,\mu,\Delta) \textrm{ lies in } \xi^\perp\} \subset \R^2$$ as a subset of the $(\mu,\Delta)$-plane.  A Riemann-Roch calculation shows $Q_\xi$ is the vertical line $$Q_\xi:\mu=-\frac{\chi}{d}.$$ It therefore intersects the line $\Delta = \frac{1}{2}$ in the rational point $(-\frac{\chi}{d},\frac{1}{2})$.  The corresponding exceptional slope to $\xi$ is the slope $\alpha\in \F E$ with $-\frac{\chi}{d}\in I_\alpha$.  The orthogonal invariants $(\mu^+,\Delta^+)$ are defined in the same way as in the general case. In Figure \ref{figure-arrangementsRank0} we display the possible arrangements of the curves involved.

Once the orthogonal invariants have been defined, there is no major difficulty in carrying out  arguments similar to those given for the general case.  The necessary minor technical changes were already considered in \cite{Woolf}.  All the results of \S \ref{ss-primaryEffective} hold verbatim.

\begin{figure}[htbp]
\begin{center}
\hspace{-.2in}\input{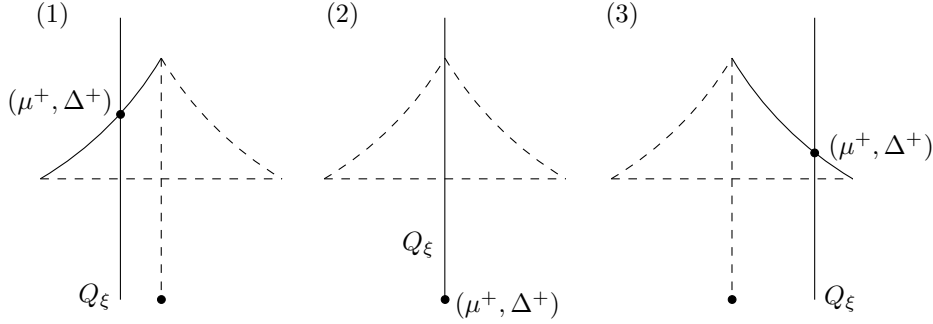}
\end{center}
\caption{Graphical depiction of the corresponding orthogonal invariants in the rank $0$ case, (1) when $(\xi,\xi_\alpha)>0$, (2) when $(\xi,\xi_\alpha)=0$, and (3) when $(\xi,\xi_\alpha)<0$.}
\label{figure-arrangementsRank0}
\end{figure}

\subsection{The secondary edge of the effective cone}\label{ssec-secondary} We now discuss the behavior of the secondary edge of the effective cone.

  When $\xi$ has small rank, the secondary edge of the effective cone exhibits special behavior, which we now recall.
\begin{itemize}
\item If $r(\xi)=0$ and $d=c_1(\xi)\geq 3$, there is a map $M(\xi)\to |\OO_{\P^2}(d)|$ sending a sheaf to its support.  This map has positive-dimensional fibers, and the pullback of $\OO(1)$ spans the secondary edge of the nef, movable, and effective cones.  See \cite{Woolf} for more information about this class.

\item If $r(\xi)=1$, then $M(\xi)$ is isomorphic to a Hilbert scheme of points in the plane.  The secondary edge of the effective cone is spanned by the exceptional divisor $B$ of the Hilbert-Chow morphism $\P^{2[n]}\to \P^{2(n)}$.  The secondary edge of the movable cone is spanned by the divisor $H$ given as the locus of schemes meeting a fixed line in the plane \cite{ABCH}.  

\item If $r(\xi)=2$, the exceptional locus $E$ of the morphism $M(\xi)\to M^{DUY}(\xi)$ is divisorial, and it spans the secondary edge of the effective cone.  Its support is the locus of singular sheaves.  The secondary edge of the movable cone coincides with the ray giving this morphism.  This case was considered in detail in \cite{Stromme}.
\end{itemize}

We note that the classes of the previous extremal effective divisors can be easily determined in terms of the isomorphism $\lambda_M: \xi^\perp \to \NS(M(\xi))$.  We concentrate on the rank $2$ case, as the lower rank cases are already well-understood \cite{Woolf,ABCH}.

\begin{proposition}
When $r(\xi)= 2$, the secondary edge of the effective cone is the ray spanned by $\lambda_M(\zeta)$, where $\zeta$ is a negative rank orthogonal Chern character with $\mu(\xi\te \zeta) = -\frac{3}{2}$.
\end{proposition}
\begin{proof}
Stromme \cite[Proposition 4.6]{Stromme} describes how to put a scheme structure on the locus $E$ of singular sheaves, at least when $M(\xi)$ admits a universal family.  When $M(\xi)$ doesn't carry a universal family, the scheme structure can still first be described locally by similar methods to the proof of Proposition \ref{BNProp}.  If $\F U/S$ is a complete family of sheaves on $\P^2$ with Chern character $\xi$, the divisor $E_S$ on $S$ of singular sheaves is cut out by the $0$th Fitting ideal of the sheaf $\F F = p_* \F Ext^1(\F U,q^* \OO_{\P^2}(-1))$ on $S$, and $\OO_S(E_S) \cong \det \F F$.  Stromme describes a canonical locally free resolution of $\F F$ of the form $$0\to A\to B\to \F F\to 0$$ where $A$ and $B$ are sheaves whose determinants are of the form $\det A = \lambda_{\F U}(\zeta_1)$ and $\det B = \lambda_{\F U}(\zeta_2)$ for some $\zeta_i\in K(\P^2)$.  An easy computation shows that $\det \F F = \lambda_{\F U}(\zeta)$, where $\zeta = \zeta_2-\zeta_1$ is an orthogonal Chern character to $\xi$ with $\OO_{M(\xi)}(E) = \lambda_M(\zeta)$, and  $\mu(\xi\te \zeta)=-\frac{3}{2}$.  
\end{proof}

Suppose $r(\xi)\geq 3$.  The complement of the open set in $M(\xi)$ consisting of stable locally-free sheaves has codimension at least $2$ by \cite[Corollary 17.1.5 and Lemma 18.3.1]{LePotierLectures}.  Thus the functor $\sHom(-,\OO_{\P^2}(-3))$ gives a birational map $M(\xi)\dashrightarrow M(\xi^D)$, where $\xi^D$ denotes the Serre dual Chern character, and this map is an isomorphism between the loci of stable locally free sheaves.  It therefore induces an isomorphism of Picard groups $\Pic(M(\xi)) \cong \Pic(M(\xi^D))$ which preserves effective cones.

However, this isomorphism does \emph{not} preserve the primary edge of the effective cone.  The $h^2$-Brill-Noether divisor $D_V$ on $M(\xi)$ is carried to the $h^0$-Brill-Noether divisor $D_{V^*}$ on $M(\xi^D)$, and vice versa.  Thus, by Corollary \ref{primaryCor}, the isomorphism takes the primary edge of the effective cone of $M(\xi^D)$ to the secondary edge of the effective cone of $M(\xi)$.  We can therefore determine the secondary edge of the effective cone of $M(\xi)$ in terms of the primary edge of the effective cone of the Serre dual moduli space.

\begin{definition}
Suppose $r(\xi)\geq 3$.  Let $(\mu_D^+,\Delta_D^+)$ be the corresponding orthogonal invariants to $\xi^D$.  The \emph{secondary corresponding orthogonal invariants to $\xi$} are $(\mu^{-},\Delta^{-})$, given by $$\mu^-= -\mu_D^+ \qquad \textrm{and} \qquad \Delta^-=\Delta_D^+.$$ The \emph{secondary corresponding exceptional slope to $\xi$} is given by $-\alpha_D$, where $\alpha_D$ is the primary corresponding exceptional slope to $\xi^D$.
\end{definition}

The following theorem is immediate.

\begin{theorem}
If $r(\xi)\geq 3$, let $\xi^{-} = (r^-,\mu^-,\Delta^-)$  where $r^{-}$ is sufficiently large and divisible.  There exist stable vector bundles $V$ of Chern character $\xi^{-}$ cohomologically orthogonal to the general $U\in M(\xi)$.  The $h^2$-Brill-Noether divisor $D_V$ spans the secondary edge of the effective cone of $M(\xi)$.
\end{theorem}

Corollary \ref{movingCor} and Theorem \ref{excEffThm} also have immediate generalizations for the secondary edges of the effective and movable cones.

\begin{example}
We close the paper with a typical computation of the effective cone of a moduli space $M(\xi)$.  

\begin{figure}[htbp]
\begin{center}
\includegraphics[scale=0.8]{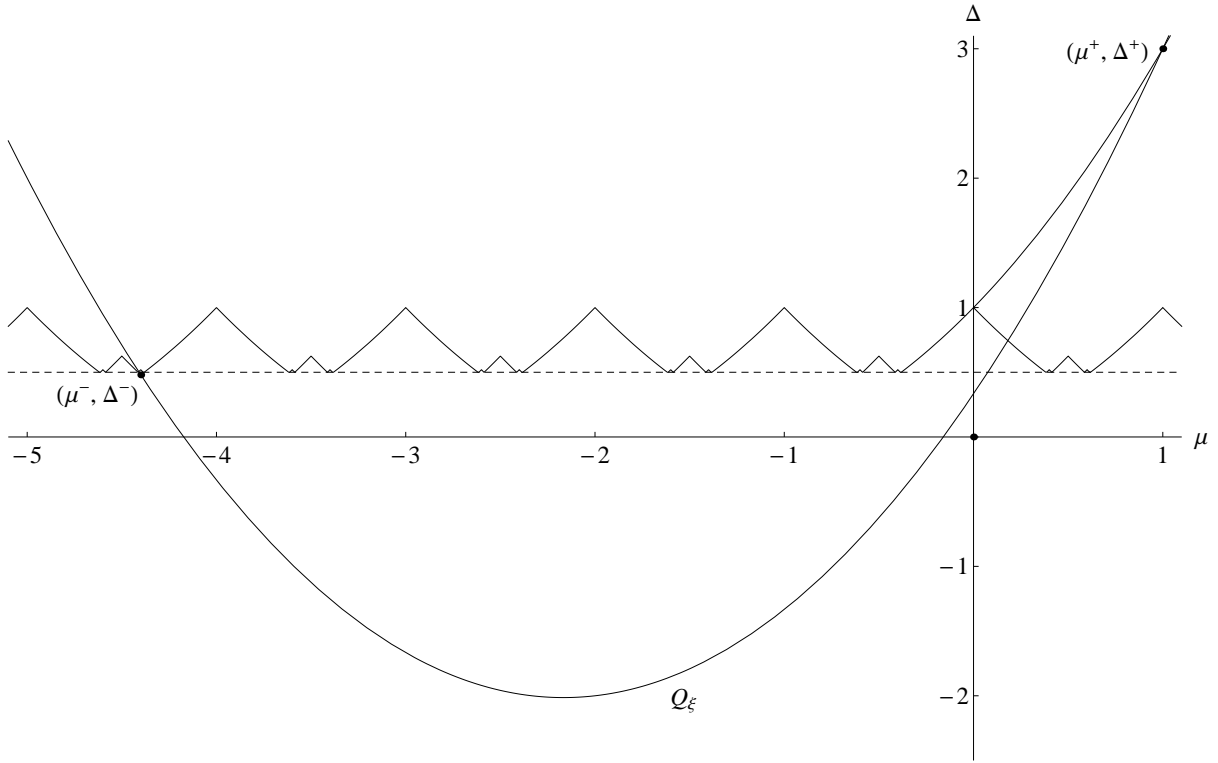}
\end{center}
\caption{Computing the corresponding orthogonal invariants for the Chern character $\xi$ of rank $3$, slope $\frac{2}{3}$, and discriminant $\frac{17}{9}$.}
\label{figure-rank3Example}
\end{figure}

Consider the Chern character $\xi = (r,\mu,\Delta) = (3,\frac{2}{3},\frac{17}{9}).$    The moduli space $M(\xi)$ has dimension $r^2(2\Delta-1)+1=26$.  The parabola $Q_\xi$ intersects the line $\Delta=\frac{1}{2}$ at the two points $(\mu_0^{\pm},\frac{1}{2})$ given by $$\mu_0^{\pm} = \frac{1}{6}(-13\pm \sqrt{181}),$$ so $$\mu_0^+ \approx 0.076 \qquad \textrm{and} \qquad \mu_0^-\approx -4.409.$$  Then $\mu_0^+\in I_0$ and $\mu_0^-\in I_{-\frac{22}{5}}$, so the primary and secondary corresponding exceptional slopes are $0$ and $-\frac{22}{5}$, respectively (See Figure \ref{figure-rank3Example}).

Consider the primary edge of the effective cone.  We have $(\xi,\xi_0)=1$, so the primary orthogonal invariants $(\mu^+,\Delta^+)$ are given by the intersection $Q_\xi\cap Q_{\xi_0}$.  We find $(\mu^+,\Delta^+) = (1,3)$.  A primary orthogonal Chern character is any character $\xi^+ = (r^+,\mu^+,\Delta^+)$ with sufficiently large and divisible rank.  There are stable vector bundles $V$ with Chern character $\xi^+$ cohomologically orthogonal to the general $U\in M(\xi)$, and Brill-Noether divisors $D_V$ span the primary edge of the effective cone.

To make the proof of the theorems explicit in this case, we observe that the general $U\in M(\xi)$ has a resolution of the form $$0\to \OO_{\P^2}(-2)^4\to \OO_{\P^2}(-1)^6\oplus \OO_{\P^2}\to U\to 0.$$ The Kronecker fibration is a rational map $\pi:M(\xi) \dashrightarrow Kr_3(4,6)$, and this space of Kronecker modules has dimension $21$ and Picard rank $1$. The divisor $D_V$ is the pullback of a multiple of the ample generator on $Kr_3(4,6)$, so it is an extremal effective divisor.

For the secondary edge of the effective cone, we observe that $Q_\xi$ passes through the invariants $\xi_{-\frac{22}{5}}$ of the exceptional bundle of slope $-\frac{22}{5}$.  The primary corresponding exceptional slope to the Serre dual character $\xi^D$ is $\frac{22}{5}$, and we find $(\mu^-,\Delta^-) = (-\frac{22}{5},\frac{12}{25})$.  Furthermore, the general $U_D\in M(\xi^D)$ admits a resolution of the form $$0\to \OO_{\P^2}(-7)\to T_{\P^2}(-6)^2\to U_D\to 0.$$  We have $\hom(\OO_{\P^2}(-7),T_{\P^2}(-6))=15,$ so the Kronecker map is a birational map $M(\xi^D) \dashrightarrow Kr_{15}(1,2)$. The Brill-Noether divisor $D_{E_{22/5}}$ spans the primary edge of the effective cone of $M(\xi^D)$.  Correspondingly, the $h^2$-Brill-Noether divisor $D_{E_{-22/5}}$ on $M(\xi)$ spans the secondary edge of the effective cone of $M(\xi)$.

\end{example}

\bibliographystyle{plain}

\end{document}